\pdfoutput=1
\documentclass{amsart}
\usepackage{enumerate}
\usepackage[osf]{mathpazo}
\usepackage{amssymb}
\usepackage{amscd}
\usepackage{mathptmx}

\newtheorem{theorem}{Theorem}[section]

\newtheorem{proposition}[theorem]{Proposition}
\newtheorem{lemma}[theorem]{Lemma}
\newtheorem{conjecture}[theorem]{Conjecture}

\theoremstyle{definition}

\newtheorem{example}[theorem]{Example}

\newtheorem{remark}[theorem]{Remark}

\DeclareMathOperator{\B}{B}%
\DeclareMathOperator{\Cl}{Cl}%
\DeclareMathOperator{\bigDiv}{Div}%
\DeclareMathOperator{\Div}{div}%
\DeclareMathOperator{\Cokernel}{Coker}%
\DeclareMathOperator{\image}{Image}%
\DeclareMathOperator{\HH}{H}%
\DeclareMathOperator{\Pic}{Pic}%
\DeclareMathOperator{\Sing}{Sing}%
\DeclareMathOperator{\Spec}{Spec} %
\DeclareMathOperator{\Proj}{Proj} %
\DeclareMathOperator{\Tor}{Tor} %

\title[Group of Units]{The Group of Units on an Affine Variety}% 

\author{Timothy J. Ford}%
\email{\tt ford@fau.edu}%
\address{Department of Mathematics, Florida Atlantic University,
  Boca Raton, Florida 33431}%
\date{\today}%
\dedicatory{Dedicated to Fook Loy} %

\begin{document}

\begin{abstract}
  The object of study is the group of units $\mathcal O^\ast(X)$ in
  the coordinate ring of a normal affine variety $X$ over an
  algebraically closed field $k$.  Methods of Galois cohomology are
  applied to those varieties that can be presented as a finite cyclic
  cover of a rational variety.  On a cyclic cover $X \rightarrow
  \mathbb A^m$ of affine $m$-space over $k$ such that the ramification
  divisor is irreducible and the degree is prime, it is shown that
  $\mathcal O^\ast(X)$ is equal to $k^\ast$, the nonzero scalars.  The
  same conclusion holds, if $X$ is a sufficiently general affine
  hyperelliptic curve.  If $X$ has a projective completion such that
  the divisor at infinity has $r$ components, then sufficient
  conditions are given for $\mathcal O^\ast(X)/k^\ast$ to be
  isomorphic to $\mathbb Z^{(r-1)}$.
\end{abstract}


\maketitle

\section{Introduction}
\label{sec:1}

\subsection{Statement of problem and summary of results}
\label{sec:1.1}
Throughout, $X$ is a normal variety defined over an algebraically
closed field $k$.  The ring of regular functions on $X$ is denoted
$\mathcal O(X)$ and the group of units, or invertible regular
functions, is denoted $\mathcal O^\ast(X)$.  If $X$ is a projective
variety in $\mathbb P^m$, the ring $\mathcal O(X)$ is always equal to
$k$ \cite[Theorem~I.3.4]{H:AG}, hence the group of units $\mathcal
O^\ast(X)$ is equal to $k^\ast$.  This confirms the intuitive notion
that on a projective variety the only functions with no zero set and
no pole set are the non-zero constants.  For this reason, as indicated
by the title, we focus on those varieties $X$ that are affine.

This article is concerned with the general question, ``What is the
group of units on an affine variety?''
Our main results tend
to be one of two types of statements, either sufficient conditions on
$X$ such that the group of units is trivial, or sufficient conditions
on $X$ such that the group of units is ``large''.  This terminology is
explained in Section~\ref{sec:1.2}.  We focus on those varieties $X$
that are either rational, or can be presented as a finite cyclic cover
of a rational variety.

The group of units on a variety $X$ is intimately connected to the
class group of Weil divisors, the Picard group, and the Brauer group.
In fact, these  connections are what led to the present study.  For
this reason, the strategy we employ is to exploit what is known about
the latter groups in order to compute the former.  Most of our
computations are derived using Galois cohomology or \'etale cohomology.

In Section~\ref{sec:2} the affine variety $X$ is a hypersurface in
$\mathbb A^{m+1}$ defined by an equation of the form $z^n = f$, where
$f$ is a polynomial in $k[x_1, \dots, x_m]$.  In
Theorems~\ref{th:2.10} and ~\ref{th:2.12}, we show that $\mathcal
O^\ast(X)$ is equal to $k^\ast$ when $f$ is irreducible and $n$ is
either a prime or is equal to $4$.  The proof relies heavily on Galois
cohomology.

In Section~\ref{sec:3} the variety $X$ is an affine open subvariety of
a finite cyclic covering $\pi : Y \rightarrow \mathbb P^m$.
Section~\ref{sec:3.1} treats the case where $\pi$ is unramified along
the divisor at infinity of $X$.  We apply these results to affine
hyperelliptic curves in Section~\ref{sec:3.2}.  If $\pi : Y
\rightarrow \mathbb P^1$ is a ramified cyclic cover of prime degree
and $k= \mathbb C$, then for a sufficiently general $Q \in \mathbb
P^1$, the group of units on the affine curve $X = \pi^{-1}(\mathbb
P^1- Q)$ is trivial (Theorem~\ref{th:3.4}).  The proofs are based on
Galois cohomology.  We see that for an affine plane curve over
$\mathbb C$ defined by an equation of the form $y^p = (x-\lambda_1)
\dotsm (x-\lambda_n)$, where $p \mid n$, the size of the group of
units depends not only on the number of points at infinity, but on the
choices of the scalars $\lambda_i$.  The proof relies on the fact that
$\mathbb C$ is uncountable and we mention here that this phenomenon
does not occur over finite ground fields.  For a nonsingular affine
curve over a finite field, the group of units always has a free
component of rank one less than the number of points at infinity (see
for example, \cite[Proposition~1.2]{MR3013032}).
Section~\ref{sec:3.3} considers the case where the ramification
divisor of $\pi$ is equal to the divisor at infinity of $X$. In this
case we show that the group of units is large
(Proposition~\ref{prop:3.12}).  The proof is based on \'etale
cohomology.  These results are applied in Section~\ref{sec:3.4} to
compute the group of units on an affine hypersurface $X$ in $\mathbb
A^m$ defined by an equation of the form $f_1 \dotsm f_r = 1$.  In
Theorem~\ref{theorem:3.15} we show that $\mathcal O^\ast(X)/k^\ast$ is
a free $\mathbb Z$-module with basis $f_1, \dots, f_{r-1}$ if the
following conditions are satisfied: $r \geq 2$, $m \geq 2$, $f_1$,
\dots, $f_r$ are distinct non-constant forms in $k[x_1, \dots, x_m]$,
and the degrees $\deg(f_1)$, \dots, $\deg(f_r)$ are relatively prime.
Results of Section~\ref{sec:3} have been applied in \cite{F:Bgars}.

After reviewing some background results, Section~\ref{sec:1} concludes
with examples meant to provide motivation for the rest of the article.

\subsection{Background results}
\label{sec:1.2}
For background definitions and theorems on homological algebra, we
refer the reader to \cite{R:IHA}. For \'etale cohomology, we suggest
\cite{M:EC}, and for all other unexplained terminology and notation,
\cite{H:AG}.

Lemma~\ref{lemma:1.1} reduces the computation of $\mathcal O^\ast(X)$
as an abstract group to the determination of the rank of a free
$\mathbb Z$-module.  Given the affine variety $X$ and a projective
completion $Y$, we see that the group of units on $X$ is closely tied
to the components of the divisor at infinity and the subgroup that
they generate in the class group $\Cl(Y)$. 
If $Y$ is a normal variety, the
divisor class group $\Cl(Y)$ is the group of Weil divisors modulo the
subgroup of principal Weil divisors
\cite[Section~II.6]{H:AG}.
The rank of the free
$\mathbb Z$-module $\mathcal O^\ast(X)/k^\ast$ is bounded above by
$r$, the number of irreducible components of the divisor at infinity.
The group $\mathcal O^\ast(X)$ is trivial, if it is equal to $k^\ast$.
The group $\mathcal O^\ast(X)$ is said to be {\em large}, if the rank
of $\mathcal O^\ast(X)/k^\ast$ is greater than or equal to $r-1$.
\begin{lemma}
  \label{lemma:1.1}
  \cite[Lemme~1]{MR0204454} If $X$ is a normal affine variety, the
  group $\mathcal O^\ast(X)/k^\ast$ is a finitely generated torsion
  free $\mathbb Z$-module.
\end{lemma}
\begin{proof}
  We sketch Samuel's proof.  Embed $X$ as an open affine subvariety of
  a normal projective variety $Y$. Write the ``divisor at infinity''
  as a union of prime divisors $Y - X = Z_1 \cup \dotsm \cup Z_r$.  As
  previously remarked, $\mathcal O^\ast(Y) = k^\ast$.  The claim
  follows from the following version of the exact sequence
  \begin{equation}
    \label{eq:101}
    1 \rightarrow
    \mathcal O^\ast(Y)
    \rightarrow
    \mathcal O^\ast(X)
    \xrightarrow{\Div}
    \bigoplus_{i=1}^r \mathbb Z Z_i
    \rightarrow
    \Cl{(Y)}
    \rightarrow
    \Cl{(X)}
    \rightarrow 0
  \end{equation}
  of Nagata (for example, \cite[Theorem~1.1]{F:Bgac}).
\end{proof}
Lemma~\ref{lemma:1.2} shows that the affine cone over a projective
variety has trivial group of units.  The proof can be found in
\cite{F:Bgac}.
\begin{lemma}
  \label{lemma:1.2}
  Let $f \in k[x_0, \dots, x_m]$ be a square-free homogeneous
  polynomial of degree $d>0$, defining a projective hypersurface $V$
  in $\mathbb P^m$.  Let $X=Z(f)$ be the affine cone over $V$ in
  $\mathbb A^{m+1}$. Then $\mathcal O^\ast(X) = k^\ast$.
\end{lemma}
In Lemma~\ref{lemma:1.3} we compute the group of units on a basic
Zariski open affine subset of $\mathbb A^m$.
\begin{lemma}
  \label{lemma:1.3}
  Let $A = k[x_1, \dots, x_m]$ be the coordinate ring of $\mathbb
  A^m$. Let $f \in A$ be non-invertible and square-free with
  factorization $f =f_1 \dotsm f_\nu$ into irreducible elements of
  $A$.  If $R = A[f^{-1}]$, then
  \begin{enumerate}[(a)]
  \item $A^\ast = k[x_1, \dots, x_m]^\ast = k^\ast$.
  \item $R^\ast = k[x_1, \dots, x_m][f^{-1}]^\ast = k^\ast \times
    \langle f_1\rangle \times \dotsm \times \langle f_\nu\rangle$.
  \end{enumerate}
\end{lemma}
\begin{proof}
  The factorization of $f$ is unique because $A$ is a unique
  factorization domain. In addition, $\Cl(A) = \Cl(R) = \langle
  0\rangle$.  The counterpart of \eqref{eq:101} for the localization
  $R = A[f^{-1}]$ is
  \begin{equation}
    \label{eq:102}
    1 \rightarrow A^\ast
    \rightarrow R^\ast
    \xrightarrow{\Div}
    \bigoplus_{i=1}^\nu \mathbb Z F_i
    \rightarrow
    \Cl(A)
    \rightarrow
    \Cl(R)
    \rightarrow 0
  \end{equation}
  where $F_i = Z(f_i)$, for $i = 1, \dots, \nu$.  The sequence splits
  because $\bigoplus \mathbb Z F_i$ is free.
\end{proof}
We will apply Theorem~\ref{th:1.4} frequently.  It is well known, but
for convenience, as well as to establish notation, we state it here.
\begin{theorem}
  \label{th:1.4}
  Let $G = \langle \sigma\rangle$ be a finite cyclic group of order
  $n$ and $M$ a $G$-module.  In $\mathbb Z G$ define elements $D =
  \sigma - 1$ and $N = 1 + \sigma + \sigma^2 +\dots + \sigma^{n-1}$.
  Let ${_N M} = \{ x \in M \mid Nx = 0\}$ denote the kernel of $N: M
  \rightarrow M$. Then
  \[
  \begin{split}
    \HH^0(G,M) &= M^G \\
    \HH^{2i-1}(G,M) &= {_N M}/D M \\
    \HH^{2i}(G,M) &= {M^G}/N M
  \end{split}
  \]
  for all $i \geq 1$.
\end{theorem}
\begin{proof}
  See \cite[Theorem~10.35]{R:IHA}, or \cite[\S~17.2]{MR2286236}.
\end{proof}

\subsubsection{Background results from \'etale cohomology}
\label{sec:1.2.1}
In this section we review some of the results from \'etale cohomology
that will be utilized in Section~\ref{sec:3.3}.  For any variety $X$
over $k$, we denote by $\mathbb G_m$ the sheaf of units for the
\'etale topology.  Then $\mathcal O^\ast(X) = \HH^0(X,\mathbb G_m)$ is
the group of global units on $X$.  We identify $\Pic{X}$, the Picard
group of $X$, with $\HH^1(X,\mathbb G_m)$
\cite[Proposition~III.4.9]{M:EC}.
The natural map $\Pic(X) \rightarrow
\Cl(X)$ is one-to-one \cite[Corollary~18.5]{F:DCG}.  If $X$ is
regular, $\Pic(X) =\Cl(X)$ \cite[Corollary~II.6.16]{H:AG}.

The torsion subgroup of $\HH^2(X, \mathbb G_m)$ is denoted
$\B^\prime(X)$ and is called the cohomological Brauer group.  If $X$
is regular, $\HH^2(X, \mathbb G_m)$ is torsion \cite[Proposition~1.4,
p.  71]{G:GBII}.  The Brauer group of classes of $\mathcal
O(X)$-Azumaya algebras is denoted $\B(X)$.  There is a natural
embedding $\B(X) \rightarrow \B^\prime(X)$ \cite[(2.1), p.
51]{G:GBI}.  In this article, the Brauer group plays only a peripheral
role. If $X$ is regular and integral with field of rational functions
$K$, then the natural map
\begin{equation}
  \label{eq:8}
  \HH^2(X, \mathbb G_m) \rightarrow \B(K)
\end{equation}
is one-to-one \cite[Corollary~1.8, p. 73]{G:GBII}.

Let $\nu > 1$ be an integer that is invertible in $k$.  The kernel of
the $\nu$th power map on $\mathbb G_m$ is $\mu_\nu$, the sheaf of
$\nu$th roots of unity.  By Kummer theory, the $\nu$th power map gives
an exact sequence of sheaves
\begin{equation}
  \label{eq:7}
  1 \rightarrow \mu_\nu \rightarrow \mathbb G_m
  \xrightarrow{\nu} \mathbb G_m \rightarrow 1
\end{equation}
for the \'etale topology on $X$ \cite[Example~II.2.18(b)]{M:EC}.  The
long exact sequence of cohomology associated to \eqref{eq:7} is
\begin{equation}
  \label{eq:12}
  \dotsm
  \rightarrow \HH^{i-1}(X, \mathbb G_m)
  \xrightarrow{\partial^{i-1}}
  \HH^{i}(X, \mu_\nu)
  \rightarrow
  \HH^{i}(X, \mathbb G_m)
  \xrightarrow{\nu}
  \HH^{i}(X, \mathbb G_m) 
  \rightarrow \dotsm\ldotp
\end{equation}
By ${_\nu\Pic(X)}$ we denote the subgroup of $\Pic(X)$ annihilated by
$\nu$. Then for $i=1$, \eqref{eq:12} gives rise to the short-exact
sequence
\begin{equation}
  \label{eq:11}
  1 \rightarrow \mathcal O^\ast(X) /
  \left(\mathcal O^\ast(X)\right)^\nu
  \rightarrow \HH^1(X,\mu_\nu)
  \rightarrow {_\nu\Pic{X}} \rightarrow 0\ldotp
\end{equation}
As shown in \cite[pp. 125--126]{M:EC}, the cohomology group $\HH^1(X,
\mu_\nu)$ classifies the Galois coverings $Y \rightarrow X$ with
cyclic Galois group $\mathbb Z/\nu$.

The proof of Proposition~\ref{lemma:3.11} utilizes the construction of
the fundamental class $s_{Z/X}$ of a prime divisor $Z$ on $X$.  We
sketch here the results that will be needed.  The reader is referred
to \cite[\S~VI.6]{M:EC} for the details. Let $X$ be a nonsingular
integral variety over $k$ and $Z \subseteq X$ a nonsingular integral
closed subvariety of codimension one. For any sheaf $F$ on $X$ the
long exact sequence of cohomology with supports in $Z$ is
\begin{multline}
  \label{eq:9}
  \dotsm \rightarrow \HH^{i-1}(X,F) \rightarrow \HH^{i-1}(X-Z,F)
  \xrightarrow{\partial^{i-1}} \\
  \HH^{i}_Z(X,F) \rightarrow \HH^{i}(X,F) \rightarrow \HH^{i}(X-Z,F)
  \rightarrow \dotsm \ldotp
\end{multline}
For $F = \mathbb G_m$, the lower degree terms of \eqref{eq:9} are
\begin{multline}
  \label{eq:10}
  1 \rightarrow \mathcal O^\ast(X) \rightarrow \mathcal O^\ast(X-Z)
  \xrightarrow{\partial^0} \\
  \HH^{1}_Z(X,\mathbb G_m) \rightarrow \Pic{(X)} \rightarrow \Pic(X-Z)
  \xrightarrow{\partial^1}  \\
  \HH^{2}_Z(X,\mathbb G_m) \rightarrow \HH^{2}(X,\mathbb G_m)
  \rightarrow \HH^{2}(X-Z,\mathbb G_m) \rightarrow \dotsm \ldotp
\end{multline}
Since $X$ is regular, $\Pic(X) = \Cl(X)$ and by \eqref{eq:101} we know
$\partial^1$ is the zero map. Using \eqref{eq:8} for both $X$ and
$X-Z$ we conclude that $\HH^{2}_Z(X,\mathbb G_m) = \langle 0\rangle$.
Comparing \eqref{eq:10} with \eqref{eq:101} we identify
$\HH^{1}_Z(X,\mathbb G_m)$ with the infinite cyclic group $\mathbb Z
Z$ and see that the generator of $\HH^{1}_Z(X,\mathbb G_m)$ maps to
the divisor class of $Z$ in $\Pic(X)$. The long exact sequence of
cohomology associated to \eqref{eq:7} simplifies to
\begin{equation}
  \label{eq:2}
  \begin{CD}
    {\HH^{1}_Z(X,\mathbb G_m)} @>{\nu}>> {\HH^{1}_Z(X,\mathbb G_m)}
    @>{\partial^1}>> {\HH^{2}_Z(X,\mu_\nu)} @>>>
    {0} \\
    @AA{=}A @AA{=}A
    @AA{=}A\\
    {\mathbb Z} @>{\nu}>> {\mathbb Z} @>>> {\mathbb Z/\nu} @>>> {0}
  \end{CD}
\end{equation}
from which we conclude $\HH^{2}_Z(X,\mu_\nu)$ is cyclic of order
$\nu$.  The generator of $\HH^{2}_Z(X,\mu_\nu)$ corresponding to the
image of $Z$ is called the {\em fundamental class of $Z$ in $X$} and is
denoted $s_{Z/X}$.

\subsection{Motivational examples}
\label{sec:1.3}
\begin{example}
  \label{example:1.4}
  Let $m >1$ and $n > 1$. Let $A = k[x_1, \dots, x_m]$ be the affine
  coordinate ring for $\mathbb A^m$.  In this example we construct an
  affine variety $X$ in $\mathbb A^m$ of degree $n$ such that the
  group $\mathcal O^\ast(X)/k^\ast$ has rank $n-1$.  Let $f_1, \dots,
  f_n$ be linear polynomials in $A$, chosen so that $f = f_1 f_2 \dots
  f_n +1$ is irreducible.  Let $X = Z(f)$, $F_1 = Z(f_1), \dots, F_n =
  Z(f_n)$ be the corresponding affine varieties in $\mathbb A^m$.
  Assume no two of the affine hyperplanes $F_1, \dots, F_n$ are
  disjoint.  Let $\mathbb P^n = \Proj{k[x_0, x_1, \dots, x_m]}$ and
  write $(\,)^\ast$ for homogenization with respect to the variable
  $x_0$.  The projective completions of $X$, $F_1$, \dots, $F_n$ are
  $\bar X = Z(f^\ast)$, $\bar F_1 = Z(f_1^\ast), \dots, \bar F_n =
  Z(f_n^\ast)$.  Let $\bar F_0 = Z(x_0)$ denote the hyperplane at
  infinity.  The complement of $X$ in $\bar X$ is the zero set of
  $x_0, f_1^\ast \dotsm f_n^\ast$.  Let $L_i = \bar F_i \cap \bar X=
  \bar F_i \cap \bar F_0$.  Then each $L_i$ is a hyperplane in $\bar
  F_0$ and $\bar X - X = L_1 +\dots + L_n$.  Using the jacobian
  criterion \cite[Theorem~I.5.1]{H:AG}, one can see that the singular
  locus of $\bar X$ agrees with the singular locus of $L_1 +\dots +
  L_n$.  But $L_1, \dots, L_n$ are distinct hyperplanes in $\bar F_0$,
  by our assumption on $f_1, \dots, f_n$.  Then $\bar X$ is regular in
  codimension one and the Serre criteria
  \cite[Proposition~II.8.23]{H:AG} show $\bar X$ is normal.  In this
  notation, sequence \eqref{eq:101} becomes
  \begin{equation}
    \label{eq:103}
    1
    \rightarrow k^\ast \rightarrow
    \mathcal O^\ast(X)
    \xrightarrow{\Div}
    \bigoplus_{i=1}^n \mathbb Z L_i
    \xrightarrow{\chi}
    \Cl(\bar X)
    \rightarrow
    \Cl(X) \rightarrow 0\ldotp
  \end{equation}
  Let $H$ denote the image of $\chi$, which is the subgroup of
  $\Cl(\bar X)$ generated by the divisors $L_1, \dots, L_n$.  We will
  prove the following.
  \begin{enumerate}[(a)]
  \item $H$ is an infinite group.
  \item $H$ is a homomorphic image of $\mathbb Z \oplus (\mathbb
    Z/n)^{(n-2)}$.
  \item The image of $\Div$ has rank $n-1$.
  \item The elements $f_1, \dots, f_{n-1}$ generate a subgroup of
    finite index in $\mathcal O^\ast(X)/k^\ast$.
  \end{enumerate}
  The free group $\mathbb Z L_i$ maps onto the subgroup of $\Cl(\bar
  X)$ generated by $L_i$. Since $L_i$ has degree one, the degree map
  $D \mapsto \deg(D) L_i$ is a splitting map.  Therefore (a) is true.
  This also shows the image of $\Div$ has rank at most $n-1$.  At the
  generic point of $\bar X$ we have
  \begin{equation}
    \label{eq:29}
    \frac{f_1^\ast}{x_0}\dotsm \frac{f_n^\ast}{x_0} = -1
  \end{equation}
  so the function $f_i^\ast/x_0$ represents a unit in the coordinate
  ring $\mathcal O(X)$.  The diagram
  \begin{equation}
    \label{eq:28}
    \begin{CD}
      1 @>>> \prod_{i=1}^n \left\langle {f_i^\ast}/{x_0} \right\rangle
      @>{\Div}>> {\bigoplus^n_{i=1} \mathbb Z \Div(f_i^\ast/x_0)}
      @>>> 0 \\
      @. @VV{\alpha}V @V{\beta}VV @VVV \\
      1 @>>> \frac{\mathcal O^\ast(X)}{k^\ast} @>{\Div}>>
      {\bigoplus^n_{i=1} \mathbb Z L_i} @>{\chi}>> H @>>> 0
    \end{CD}
  \end{equation}
  commutes, where the second row comes from \eqref{eq:103} and is
  exact.  Since $\bar F_i$ intersects $\bar X$ along $L_i$ with
  intersection multiplicity $n$, we see that the divisor of
  $f_i^\ast/x_0$ on $\bar X$ is
  \begin{equation}
    \label{eq:104}
    \Div(f_i^\ast/x_0) = nL_i - L_1 - \dots - L_n\ldotp
  \end{equation}
  The matrix of the map $\beta$ is:
  \begin{equation}
    \label{eq:26}
    C =
    \begin{bmatrix}
      n-1 & -1 & \dots & -1 & -1 \\
      -1 & n-1 & \dots & -1 & -1\\
      & &         \vdots & & \\
      -1 & -1 & \dots & n-1 & -1 \\
      -1 & -1 & \dots & -1 & n-1
    \end{bmatrix}\ldotp
  \end{equation}
  Using row and column operations, one can compute the invariant
  factors of $C$. They are $1$ and $0$, each with multiplicity one,
  and $n$ with multiplicity $n-2$.  The Snake Lemma
  \cite[Theorem~6.5]{R:IHA} applied to \eqref{eq:28} gives the exact
  sequence
  \begin{equation}
    \label{eq:27}
    0 \rightarrow \Cokernel{\alpha}
    \rightarrow \Cokernel{\beta}
    \rightarrow H \rightarrow 0
  \end{equation}
  of finitely generated abelian groups.  Since $\Cokernel{\beta} \cong
  \mathbb Z\oplus (\mathbb Z/n)^{(n-2)}$, (b) follows from
  \eqref{eq:27}.  Since $H$ is infinite, \eqref{eq:27} implies the
  torsion free rank of $H$ is one.  The exact functor $()\otimes_Z
  \mathbb Q$ applied to the bottom row of \eqref{eq:28} gives (c).
  The exact functor $()\otimes_Z \mathbb Q$ applied to \eqref{eq:27}
  shows $\Cokernel{\alpha}$ is finite. Together with \eqref{eq:29},
  this gives (d).
  
  We have shown that $\mathcal O^\ast(X)/k^\ast$ is a torsion-free
  abelian group of rank $n-1$ and the elements $f_1, \dots, f_{n-1}$
  generate a subgroup of finite index.  We ask whether the elements
  $f_1, \dots, f_{n-1}$ generate the group $\mathcal
  O^\ast(X)/k^\ast$.  We do not know the general answer to this
  question.  Examples~\ref{example:1.5}, 
  \ref{example:1.9}, \ref{example:1.10}, \ref{example:1.6},
  \ref{example:1.11} prove that
  the answer is yes for some particular cases.  Another affirmative
  answer is given in Example~\ref{example:3.14}, where we impose the
  additional hypotheses on $X$ that $\bar X \rightarrow \mathbb
  P^{m-1}$ is cyclic, and the ramification divisor is $\bar X -X$, the
  divisor at infinity.
\end{example}
\begin{example}
  \label{example:1.5}
  In the notation of Example~\ref{example:1.4}, let $f_i = x_i$. If $n
  \leq m$, and $f = x_1 \dotsm x_n + 1$, then in $\mathcal O(X)$ we
  have $x_1 = -x_2^{-1} \dotsm x_{n}^{-1}$.  The isomorphism
  \[
  \mathcal O(X) = \frac{k[x_1, \dots, x_m]}{(x_1 \dotsm x_n + 1)}
  \cong k[x_2, \dots, x_{n}][x_2^{-1}, \dots, x_{n}^{-1}]
  \]
  is defined by eliminating $x_1$.  By Lemma~\ref{lemma:1.3},
  $\mathcal O^\ast(X)$ is equal to $k^\ast\times\langle f_2\rangle
  \times\dotsm \times \langle f_{n}\rangle$ which by symmetry is equal
  to $k^\ast\times\langle f_1\rangle \times\dotsm \times \langle
  f_{n-1}\rangle$.
\end{example}
\begin{example}
  \label{example:1.9}
  In the notation of Example~\ref{example:1.4}, assume $n = m$.
  Moreover, assume $f_1, \dots, f_n$ are linear polynomials in $k[x_1,
  \dots, x_n]$ such that the hyperplanes $F_1 = Z(f_1), \dots, F_n=
  Z(f_n)$ in $\mathbb A^n$ are in general position.  After an affine
  change of coordinates \cite[\S~2.3, p. 40]{MR0313252}, we reduce to
  the case of Example~\ref{example:1.5}.  Therefore, $f_1, \dots,
  f_{n-1}$ make up a free basis for the $\mathbb Z$-module $\mathcal
  O^\ast(X)/k^\ast$.
\end{example}
\begin{example}
  \label{example:1.10}
  In the notation of Example~\ref{example:1.4}, suppose $n=2$ and $m
  \geq 2$. Let $f_1$, $f_2$ be linear polynomials in $k[x_1, x_2,
  \dots, x_m]$ such that the hyperplanes $F_1 = Z(f_1)$, and $F_2=
  Z(f_2)$ in $\mathbb A^n$ have nontrivial intersection.  Let $X$ be
  the hypersurface in $\mathbb A^m$ defined by $f_1f_2+1 = 0$. After
  an affine change of coordinates, we can assume $f_1 = x_1$, $f_2=
  x_2$. Therefore $\mathcal O (X)$ is isomorphic to $R[x_3, \dots,
  x_m]$, where $R = k[x_1, x_2]/(x_1x_2+1)$. By
  Example~\ref{example:1.9}, the group $\mathcal O^\ast(X)$ is equal
  to $k^\ast\times\langle f_1\rangle$.
\end{example}
\begin{example}
  \label{example:1.6}
  In the notation of Example~\ref{example:1.4}, suppose $n=3$ and
  $m=2$. Let $f_1$, $f_2$, $f_3$ be linear polynomials in $k[x_1,
  x_2]$ defining three lines $F_1 = Z(f_1)$, $F_2=Z(f_2)$, $F_3=
  Z(f_3)$ in $\mathbb A^2$, no two of which are parallel.  Let $X$ be
  the affine curve in $\mathbb A^2$ defined by $f_1f_2f_3 +1 = 0$.  In
  this case, $\bar X$ is the nonsingular cubic curve in $\mathbb P^2$
  defined by $f_1^\ast f_2^\ast f_3^\ast +x_0^3 = 0$.  The three
  points on $\bar X$ where $x_0=0$ are denoted $L_1$, $L_2$, $L_3$.
  They generate the group $H$ in \eqref{eq:28}.  Notice that
  \eqref{eq:104} implies $H$ is generated by $L_3$ and $L_1-L_3$.
  Since the genus of $\bar X$ is one, $L_1-L_3$ is not principal
  \cite[Example~II.6.10.1]{H:AG}.  Therefore, $H$ is equal to the
  internal direct sum $\mathbb Z L_3 \oplus (\mathbb Z/3) (L_1 -
  L_3)$.  It follows that $\Cokernel{\alpha} = \langle 0\rangle$ in
  \eqref{eq:27}.  This implies the group $\mathcal O^\ast(X)$ is equal
  to $k^\ast\times\langle f_1\rangle \times\langle f_2\rangle$.
\end{example}
\begin{example}
  \label{example:1.11}
  In the notation of Example~\ref{example:1.4}, suppose $n=3$ and $m
  \geq 2$. Then the group $\mathcal O^\ast(X)$ is equal to
  $k^\ast\times\langle f_1\rangle \times\langle f_2\rangle$. This is
  because by an affine change of coordinates we can reduce to the case
  where $\mathcal O (X)$ is a polynomial ring over either the ring in
  Example~\ref{example:1.9}, or the ring in Example~\ref{example:1.6}.
\end{example}
\begin{example}
  \label{example:1.7}
  A class of curves that have been widely studied are the affine
  Fermat curves $F = Z(x^n + y^n - 1)$ in $\mathbb A^2$, where $n \geq
  2$. Assume $n$ is invertible in $k$. Since $k$ is algebraically
  closed, $x^n + y^n$ factors in $k[x,y]$ into a product of $n$
  distinct linear forms.  Then $F$ can be viewed as $Z(f_1\dotsm f_n
  +1)$ and the argument in Example~\ref{example:1.4} shows that the
  group $\mathcal O^\ast(F)/k^\ast$ is free of rank ${n-1}$.  Let
  $\bar F$ denote the projective completion of $F$, and $X$ the affine
  open subset of $F$ where $x y \neq 0$.  Let $H$ denote the subgroup
  of $\Cl(\bar F)$ generated by the $3 n$ points in $\bar F - X$.  If
  $H_0$ is the kernel of the degree homomorphism, then
  \begin{equation}
    \label{eq:106}
    0 \rightarrow H_0 \rightarrow H \xrightarrow{\deg}
    \mathbb Z \rightarrow 0
  \end{equation}
  is a split exact sequence, because a point has degree one.  As shown
  in \cite{MR563965}, $H_0$ is a finite group which is annihilated by
  $n$.  Using Lemma~\ref{lemma:1.1} we find that $\mathcal
  O^\ast(X)/k^\ast$ has rank $3n-1$.  A basis for $\mathcal
  O^\ast(X)/k^\ast$ is computed in \cite{MR0441978}.
\end{example}
\section{A cyclic cover of affine space}
\label{sec:2}
In this section we consider an affine variety $X$ which is a finite
cyclic cover of $\mathbb A^m$.  It will be helpful to set up some
notation.  Let $A = k[x_1, \dots, x_m]$ be the coordinate ring of
$\mathbb A^m$.  Let $f$ be a non-invertible square-free element of
$A$.  Assume $n \geq 2$ is invertible in $k$ and $\zeta$ is a
primitive $n$th root of unity in $k$.  Let $T = A[z]/(z^n - f)$ denote
$\mathcal O(X)$, the coordinate ring of the affine variety $X$.  Then
$T$ is a ramified cyclic extension of $A$, and the associated morphism
on varieties is $\pi: X \rightarrow \mathbb A^m$.  Let $f = f_1 \dots
f_\nu$ be a factorization of $f$ into irreducibles in $A$.  By
Eisenstein's criterion, for instance, $T$ is an integral domain.  The
singular locus of $X$ corresponds to the singular locus of the affine
hypersurface $F = Z(f) \subseteq \mathbb A^m$.  Hence $X$ is regular
in codimension one. By the Serre criteria, $X$ is normal, and $T$ is
an integrally closed integral domain.  The assignment $\sigma (z) =
\zeta z$ defines an $A$-algebra automorphism of $T$.  If $G = \langle
\sigma\rangle$, then $T^G = A$.  The morphism $\pi$ ramifies only
above the divisor $F$.  If we let $R = A[f^{-1}]$ and $S = T[z^{-1}]$,
then $S$ is a Galois extension of $R$ with cyclic group $G$.  The
rings defined so far make up this commutative diagram
\begin{equation}
  \label{eq:107}
  \begin{CD}
    T=A[\sqrt[n]{f}] @>>>   S=R[\sqrt[n]{f}] \\
    @AAA @AAA \\
    A @>>> R = A[f^{-1}]
  \end{CD}
\end{equation}
where an arrow represents set inclusion.  The norm
\begin{equation}
  \label{eq:108}
  N : T^\ast \rightarrow k^\ast  
\end{equation}
is a homomorphism of abelian groups, defined by $N(u) = u \sigma(u)
\dotsm \sigma^{n-1}(u)$.  If $a \in k^\ast$, then $N(a) = a^n$ because
$k$ is a subfield of $A$. Since $k$ is algebraically closed, $N$ is
onto.
\subsection{The group of units on a cyclic cover of affine space}
\label{sec:2.1}
The notation established in the preceding paragraph is used throughout
this section.
\begin{proposition}
  \label{prop:2.1}
  In the above context, the following are true.
  \begin{enumerate}[(a)]
  \item $\displaystyle{\HH^i(G,k^\ast) =
      \begin{cases}
        k^\ast
        & \text{if $i=0$}\\
        \mu_n & \text{if $i=1, 3, 5, \dots$}\\
        \langle 1\rangle & \text{if $i=2, 4, 6, \dots$}
      \end{cases}}$
  \item $\displaystyle{\HH^i(G,R^\ast) =
      \begin{cases}
        R^\ast = k^\ast \times \langle f_1\rangle \times \dotsm \times
        \langle f_\nu\rangle
        & \text{if $i=0$}\\
        \mu_n & \text{if $i=1, 3, 5, \dots$}\\
        \frac{\langle f_1\rangle}{\langle f_1^n\rangle} \times \dotsm
        \times \frac{\langle f_\nu\rangle}{\langle f_\nu^n\rangle} &
        \text{if $i=2, 4, 6, \dots$}
      \end{cases}}$
  \item $\displaystyle{\HH^i(G,R^\ast/k^\ast) =
      \begin{cases}
        R^\ast/k^\ast = \langle f_1\rangle \times \dotsm \times
        \langle f_\nu\rangle
        & \text{if $i=0$}\\
        \langle 1\rangle &
        \text{if $i=1, 3, 5, \dots$}\\
        \frac{\langle f_1\rangle}{\langle f_1^n\rangle} \times \dotsm
        \times \frac{\langle f_\nu\rangle}{\langle f_\nu^n\rangle} &
        \text{if $i=2, 4, 6, \dots$}
      \end{cases}}$
  \end{enumerate}
\end{proposition}
\begin{proof}
  The sequence of trivial $G$-modules
  \begin{equation}
    \label{eq:109}
    1 \rightarrow k^\ast  \rightarrow R^\ast \rightarrow R^\ast/k^\ast
    \rightarrow 1
  \end{equation}
  is exact. Use Lemma~\ref{lemma:1.3} and Theorem~\ref{th:1.4}.
\end{proof}
\begin{lemma}
  \label{lemma:2.2}
  In the context of Section~\ref{sec:2.1}, let $H$ be a subgroup of
  $G$ of order $h$.  If $(T^\ast)^H = k^\ast$, then the following are
  true.
  \begin{enumerate}[(a)]
  \item $(T^\ast/k^\ast)^H = \langle 1\rangle$.
  \item $\HH^{2i}(H,T^\ast/k^\ast) = \langle 1\rangle$, for all $i
    >0$.
  \item There is an exact sequence
    \[
    1 \rightarrow \mu_h \rightarrow \HH^{2i-1}(H,T^\ast) \rightarrow
    \HH^{2i-1}(H, T^\ast/k^\ast) \rightarrow 1
    \]
    for all $i >0$.
  \end{enumerate}
\end{lemma}
\begin{proof}
  Begin with the exact sequence
  \begin{equation}
    \label{eq:110}
    1 \rightarrow k^\ast
    \rightarrow T^\ast
    \rightarrow T^\ast/k^\ast
    \rightarrow 1\ldotp
  \end{equation}
  By Lemma~\ref{lemma:1.1}, $T^\ast/k^\ast$ is a finitely generated
  torsion free $\mathbb Z$-module.  The long exact sequence associated
  to \eqref{eq:110} is
  \begin{multline}
    \label{eq:111}
    1 \rightarrow k^\ast \xrightarrow{\alpha^0} (T^\ast)^H \rightarrow
    (T^\ast/k^\ast)^H
    \xrightarrow{\partial^0} \\
    \HH^1(H,k^\ast) \xrightarrow{\alpha^1} \HH^1(H,T^\ast) \rightarrow
    \HH^1(H, T^\ast/k^\ast)
    \xrightarrow{\partial^1} \\
    \HH^2(H,k^\ast) \xrightarrow{\alpha^2} \HH^2(H,T^\ast) \rightarrow
    \HH^2(H, T^\ast/k^\ast) \xrightarrow{\partial^2} \dots
  \end{multline}
  By hypothesis, $\alpha^0$ is an isomorphism, so $\partial^0$ is
  one-to-one.  Since $G$ acts trivially on $k^\ast$, $\HH^1(H,k^\ast)
  = \mu_h$, a finite group.  Because $(T^\ast/k^\ast)^H$ is a subgroup
  of the torsion free group $T^\ast/k^\ast$, we conclude that (a) is
  true.  Because $H$ is a cyclic group, part (b) follows from
  Theorem~\ref{th:1.4} and part (a).  By
  Proposition~\ref{prop:2.1}(a), for $i > 0$, $\HH^{2i}(H,k^\ast) =
  \langle 1 \rangle$. Part (c) follows from \eqref{eq:111}, part (a)
  and periodicity.
\end{proof}
\begin{theorem}
  \label{th:2.3}
  In the context of Section~\ref{sec:2.1}, assume $n= p$ is a prime
  number.  There is an isomorphism of $\mathbb Z[G]$-modules
  $T^\ast/k^\ast \cong A_1 \oplus \dotsm \oplus A_t$ where $A_1,
  \dots, A_t$ are $\mathbb Z[\zeta]$-ideals in $\mathbb Q[\zeta]$. It
  follows that $T^\ast/k^\ast$ is a free $\mathbb Z$-module of rank
  $(p-1) t$ and there are isomorphisms
  \[
  \begin{split}
    \HH^1(G,T^\ast) & \cong \mu_p \times  (\mathbb Z/p)^{(t)}\\
    \HH^1(G,T^\ast/k^\ast) & \cong (\mathbb Z/p)^{(t)}
  \end{split}
  \]
\end{theorem}
\begin{proof}
  Since $(T^\ast)^G = k^\ast$, this follows from Lemma~\ref{lemma:2.2}
  and the structure theory for a module over a cyclic group of order
  $p$, \cite[Theorem~(74.3)]{CR:RTF}. A similar argument is given in
  \cite[Proposition~2.17]{F:Bgrccs}.
\end{proof}
\begin{example}
  \label{example:2.4}
  When $\nu > 1$, the group of units of $T$ can be trivial.  For
  instance, let $f = (xy-1) (x-\alpha_1) \dotsm (x-\alpha_d)$, where
  $\alpha_1, \dots, \alpha_d$ are distinct elements in $k^\ast$.  As
  computed in \cite[\S 3.3]{F:rBgadp}, the group of units in
  $T=A[z]/(z^2 -f)$ is $k^\ast$.
\end{example}
The group $G$ acts as a group of automorphisms of $\Cl(T)$.  This
action is induced on the group of divisors $\bigDiv(T)$ by sending a
height one prime ideal $I$ to its conjugate $\sigma(I)$.  By this same
action, $G$ acts as a group of automorphisms of $\Pic(S)$.  Because
$\Spec{S}$ is nonsingular, $\Pic{(S)} = \Cl{(S)}$.  Associated to the
Galois extension $S/R$ is
\begin{multline}
  \label{eq:112}
  1 \rightarrow \HH^{1}(G,{S^\ast}) \xrightarrow{\alpha_1} \Pic(R)
  \xrightarrow{\alpha_2} \left(\Pic{S}\right)^G
  \xrightarrow{\alpha_3} \\
  \HH^2(G, S^\ast) \xrightarrow{\alpha_4} \B(S/R)
  \xrightarrow{\alpha_5} \HH^{1}(G,\Pic{S}) \xrightarrow{\alpha_6}
  \HH^{3}(G,{S^\ast})
\end{multline}
the so-called Chase-Harrison-Rosenberg exact sequence of cohomology
groups \cite[Corollary~5.5]{CHR:GtGccr}.  The group $\B(S/R)$
appearing in \eqref{eq:112} is the kernel of the natural map on Brauer
groups $\B(R) \rightarrow \B(S)$. In this article it plays only a
peripheral role.  Since $\Pic{R} = 0$, sequence \eqref{eq:112} and
periodicity implies
\begin{equation}
  \label{eq:113}
  \HH^i(G, S^\ast) = \langle 1\rangle
\end{equation}
if $i > 0$ is odd.
\begin{proposition}
  \label{prop:2.5}
  In the context of Section~\ref{sec:2.1}, the following are true.
  \begin{enumerate}[(a)]
  \item If $i \geq 0$ is even, then
    \[
    \HH^i(G,S^\ast/R^\ast) = (S^\ast/R^\ast)^G \cong \mathbb Z/n
    \]
    is generated by the coset containing $z$.
  \item If $i > 0$ is odd, there is an exact sequence
    \[
    1 \rightarrow \HH^i(G, S^\ast/R^\ast) \rightarrow \HH^{i+1}(G,
    R^\ast) \rightarrow \HH^{i+1}(G, S^\ast) \rightarrow 1
    \]
    of $\mathbb Z/n$-modules.
  \end{enumerate}
\end{proposition}
\begin{proof}
  The exact sequence of $G$-modules
  \begin{equation}
    \label{eq:114}
    1 \rightarrow R^\ast
    \rightarrow S^\ast
    \rightarrow S^\ast/R^\ast \rightarrow 1
  \end{equation}
  gives rise to the exact sequence of cohomology
  \begin{multline}
    \label{eq:115}
    1 \rightarrow R^\ast \rightarrow (S^\ast)^G \rightarrow
    (S^\ast/R^\ast)^G \xrightarrow{\partial^0} \HH^1(G, R^\ast)
    \rightarrow
    \HH^1(G, S^\ast) \rightarrow \\
    \HH^1(G, S^\ast/R^\ast) \xrightarrow{\partial^1} \HH^2(G, R^\ast)
    \rightarrow
    \HH^2(G, S^\ast) \rightarrow \\
    \HH^2(G, S^\ast/R^\ast) \xrightarrow{\partial^2} \HH^3(G, R^\ast)
    \rightarrow \HH^3(G, S^\ast) \rightarrow \dots \ldotp
  \end{multline}
  Equation \eqref{eq:113} implies $\partial^0$ and $\partial^2$ are
  onto.  Since $R^\ast = (S^\ast)^G$, $\partial^0$ is an isomorphism.
  By Proposition~\ref{prop:2.1}, $\HH^j(G,R^\ast)$ is a cyclic group
  of order $n$, if $j$ is odd.  The minimum polynomial of $z$ is $z^n
  - f$. In $S^\ast/R^\ast$, the coset containing $z$ has order $n$.
  This proves $(S^\ast/R^\ast)^G$ is cyclic of order $n$ and is
  generated by the coset containing $z$.  The image of the norm map $N
  : S^\ast/R^\ast \rightarrow S^\ast/R^\ast$ is $\langle 1 \rangle$.
  By Theorem~\ref{th:1.4}, if $i$ is even, $\HH^i(G, S^\ast/R^\ast) =
  (S^\ast/R^\ast)^G$.  This proves (a). Since $\partial^2$ is also an
  isomorphism, we get (b).
\end{proof}
\begin{proposition}
  \label{prop:2.6}
  In the context of Section~\ref{sec:2.1}, the following are true
  \begin{enumerate}[(a)]
  \item $\displaystyle{(S^\ast/k^\ast)^G = \langle z\rangle \times
      \langle f_2\rangle \times \dotsm \times \langle f_\nu\rangle}$
    is a free $\mathbb Z$-module of rank $\nu$.
  \item If $i > 0$ is odd, then
    \[
    \HH^i(G, S^\ast/k^\ast) = \langle 1\rangle\ldotp
    \]
  \item If $i > 0$ is even, there is an exact sequence
    \[
    1 \rightarrow \HH^i(G, S^\ast) \rightarrow \HH^i(G, S^\ast/k^\ast)
    \rightarrow \mathbb Z/n \rightarrow 1
    \]
    of $\mathbb Z/n$-modules.
  \end{enumerate}
\end{proposition}
\begin{proof}
  In the exact sequence of $G$-modules
  \begin{equation}
    \label{eq:116}
    1 \rightarrow R^\ast/k^\ast
    \rightarrow S^\ast/k^\ast
    \xrightarrow{\eta} S^\ast/R^\ast \rightarrow 1
  \end{equation}
  $\eta(z) = z$.  The sequence of cohomology associated to
  \eqref{eq:116} is
  \begin{equation}
    \label{eq:117}
    1  \rightarrow R^\ast/k^\ast
    \rightarrow (S^\ast/k^\ast)^G
    \xrightarrow{\eta}
    (S^\ast/R^\ast)^G
    \xrightarrow{\partial^0}
    \HH^1(G, R^\ast/k^\ast)
    \rightarrow     \dots \ldotp
  \end{equation}
  By Proposition~\ref{prop:2.1}, in \eqref{eq:117}, $\eta$ is onto. By
  Proposition~\ref{prop:2.5}, the image of $\eta$ is generated by
  $\eta(z)$.  So $(S^\ast/k^\ast)^G$ is an extension of
  $R^\ast/k^\ast$ by the finite cyclic group $\langle z\rangle$.  By
  Lemma~\ref{lemma:1.1}, $(S^\ast/k^\ast)^G$ is a finitely generated
  torsion free $\mathbb Z$-module.  By Lemma~\ref{lemma:1.3},
  $R^\ast/k^\ast = \langle f_1\rangle \times \dotsm \times \langle
  f_\nu\rangle$.  Since $z^n = f_1 \dotsm f_\nu$, this implies
  $(S^\ast/k^\ast)^G$ is generated by $z, f_2, \dots, f_\nu$, proving
  (a).
  
  The exact sequence of $G$-modules
  \begin{equation}
    \label{eq:118}
    1 \rightarrow k^\ast  \rightarrow S^\ast \rightarrow S^\ast/k^\ast
    \rightarrow 1
  \end{equation}
  gives the exact sequence of cohomology
  \begin{multline}
    \label{eq:119}
    \HH^{2i-1}(G, S^\ast) \rightarrow \HH^{2i-1}(G, S^\ast/k^\ast)
    \xrightarrow{\partial^{2i-1}}
    \HH^{2i}(G, k^\ast) \\
    \rightarrow \HH^{2i}(G, S^\ast) \rightarrow \HH^{2i}(G,
    S^\ast/k^\ast) \xrightarrow{\partial^{2i}} \HH^{2i+1}(G, k^\ast)
    \rightarrow \HH^{2i+1}(G, S^\ast) \rightarrow
  \end{multline}
  where $i > 0$ is arbitrary.  Parts (b) and (c) follow from
  \eqref{eq:119}, \eqref{eq:113}, and Proposition~\ref{prop:2.1}(a).
\end{proof}
\begin{example}
  When $\nu > 1$, the group of units of $T$ can be non-trivial. For
  instance, let $f = (xy-1)(xy+1) = f_1 f_2$.  As computed in
  \cite[\S~3.2]{F:rBgadp}, the group of units in $T=A[z]/(z^2 -f)$ is
  \begin{equation}
    \label{eq:120}
    T^\ast = k^\ast
    \times \langle z-xy\rangle
  \end{equation}
  and the group of units in $S = T[z^{-1}]$ is
  \begin{equation}
    \label{eq:121}
    S^\ast =  k^\ast
    \times \langle z-xy\rangle
    \times \langle z-xy+1\rangle
    \times \langle z-xy -1\rangle\ldotp 
  \end{equation}
  Use the identity $\sigma(z-xy) = (z-xy)^{-1}$ to compute
  \begin{equation}
    \label{eq:122}
    \HH^i(G,T^\ast) =
    \begin{cases}
      k^\ast & \text{if $i=0$}\\
      \mu_2 \times \frac{\langle z-xy\rangle}{\langle (z-xy)^2\rangle}
      & \text{if $i=1, 3, 5, \dots$}\\
      \langle 1\rangle & \text{if $i=2, 4, 6, \dots$}
    \end{cases}
  \end{equation}
  and
  \begin{equation}
    \label{eq:123}
    \HH^i(G,T^\ast/k^\ast) =
    \begin{cases}
      \langle 1\rangle & \text{if $i$ is even}             \\
      \frac{\langle z-xy\rangle}{\langle (z-xy)^2\rangle} & \text{if
        $i$ is odd.}
    \end{cases}
  \end{equation}
  Use the identities
  \[
  \begin{split}
    \sigma(z-xy+1) &= (z-xy+1) (z- xy)^{-1} \\
    \sigma(z-xy-1) &= -(z-xy-1) (z- xy)^{-1} \\
    2 z &=   ( z-xy+1)(z-xy-1) (z- xy)^{-1}    \\
    -2 (xy-1) &=   ( z-xy+1)^2 (z- xy)^{-1}    \\
    -2 (xy+1) &= (z-xy-1)^2 (z- xy)^{-1} % \\
    % ( z-xy+1)(z+xy+1) &= 2 z \\
    % ( z-xy-1)(z+xy-1) &= -2 z
  \end{split}
  \]
  to compute
  \begin{equation}
    \label{eq:124}
    \HH^i(G,S^\ast) =
    \begin{cases}
      R^\ast = k^\ast \times \langle xy-1\rangle\times\langle
      xy+1\rangle
      & \text{if $i=0$}\\
      \langle 1\rangle & \text{if $i>0$}
    \end{cases}
  \end{equation}
  and
  \begin{equation}
    \label{eq:125}
    \HH^i(G,S^\ast/k^\ast) =
    \begin{cases}
      \langle z\rangle\times\langle xy+1\rangle
      & \text{if $i=0$}\\
      \langle 1\rangle & \text{if $i=1, 3, 5, \dots$}\\
      \frac{\langle z\rangle}{\langle z^2\rangle} & \text{if $i=2, 4,
        6, \dots$.}
    \end{cases}
  \end{equation}
  These results agree with Propositions~\ref{prop:2.5} and
  ~\ref{prop:2.6}.
\end{example}

\subsection{The ramification divisor is irreducible}
\label{sec:2.2}
In addition to the notation established in the opening paragraph of
Section~\ref{sec:2}, assume that $T = A[z]/(z^n-f)$, where $f$ is
irreducible in $A$.
\begin{theorem}
  \label{th:2.8}
  In the context above, the following are true.
  \begin{enumerate}[(a)]
  \item $\Cl(T) = \Cl(S)$.
  \item $\Cl(T)^G = \Cl(S)^G = \langle 0\rangle$.
  \item
    \[
    \HH^i(G, T^\ast) =
    \begin{cases}
      k^\ast & \text{if $i=0$}\\
      \cong \mathbb Z/n & \text{if $i=1, 3, 5, \dots$}\\
      \langle 1\rangle & \text{if $i=2, 4, 6, \dots$}
    \end{cases}
    \]
  \item
    \[
    \HH^i(G, S^\ast) =
    \begin{cases}
      R^\ast =
      k^\ast\times \langle f\rangle & \text{if $i=0$}\\
      \langle 1\rangle & \text{if $i>0$}
    \end{cases}
    \]
  \item $\B(S/R) \cong \HH^1(G, \Pic{S})$.
  \end{enumerate}
\end{theorem}
\begin{proof}
  Since $T/(z) = A/(f)$, the ideal $I = Tz$ is a height one prime.  In
  $T$, $z^n = f$, so the ideal $T f$ has only one minimal prime,
  namely the height one prime $I = Tz$.  By Nagata's Theorem
  \cite[Theorem~7.1]{F:DCG}, the sequence
  \begin{equation}
    \label{eq:126}
    1 \rightarrow T^\ast
    \rightarrow S^\ast
    \xrightarrow{\Div}
    \mathbb Z I
    \rightarrow \Cl(T) \rightarrow \Cl(S) \rightarrow 0
  \end{equation}
  is exact.  The divisor of $z$ is $\Div(z) =I$, so \eqref{eq:126}
  shows $\Cl(T) = \Cl(S)$. It also follows from \eqref{eq:126} that
  $S^\ast/T^\ast = \langle z\rangle$. Since $\sigma(z) = \zeta z$ and
  $\zeta \in T^\ast$, we see that $S^\ast/T^\ast$ is a trivial
  $G$-module.  Therefore,
  \begin{equation}
    \label{eq:127}
    \HH^i(G,  S^\ast/T^\ast) =
    \begin{cases}
      \langle z\rangle  & \text{if $i=0$,}\\
      \langle 1\rangle  & \text{if $i = 1, 3, 5, \dots$,}\\
      \langle z\rangle/ \langle z^n\rangle & \text{if $i = 2, 4, 6,
        \dots$}
    \end{cases}
  \end{equation}
  by Theorem~\ref{th:1.4}.  The long exact sequence associated to $1
  \rightarrow T^\ast \rightarrow S^\ast \rightarrow S^\ast/T^\ast
  \rightarrow 1$ is
  \begin{equation}
    \label{eq:128}
    1 \rightarrow (T^\ast)^G \rightarrow (S^\ast)^G \rightarrow
    S^\ast/T^\ast 
    \rightarrow \HH^1(G, T^\ast) \rightarrow \HH^1(G, S^\ast)
    \rightarrow
    \HH^1(G, S^\ast/T^\ast) \rightarrow\dots
  \end{equation}
  As in Lemma~\ref{lemma:1.3}, $\Pic(R) = \langle 0\rangle$.  By
  sequence \eqref{eq:112}, $\HH^i(G,S^\ast) = \langle 1\rangle$ for
  odd $i$.  Again by Lemma~\ref{lemma:1.3}, $(T^\ast)^G = A^\ast=
  k^\ast$.  The norm map \eqref{eq:108} is onto, so for $i = 2, 4, 6,
  \dots$, we have $\HH^i(G, T^\ast) = 1$. By Lemma~\ref{lemma:1.3},
  $(S^\ast)^G = R^\ast = k^\ast \times \langle f\rangle$. The terms of
  lowest degree in \eqref{eq:128} give rise to the short exact
  sequence
  \begin{equation}
    \label{eq:129}
    1  \rightarrow  \langle f\rangle
    \rightarrow  \langle z\rangle
    \rightarrow \HH^1(G, T^\ast) \rightarrow 1\ldotp
  \end{equation}
  Therefore $\HH^i(G, T^\ast) =\langle z\rangle/\langle z^n\rangle
  \cong \mathbb Z/n$ for odd $i$, proving (c).  On the element $z$,
  the norm map $N: S^\ast \rightarrow R^\ast$ is
  \begin{equation}
    \label{eq:130}
    N(z) = z \zeta z \dotsm \zeta^{n-1} z = \zeta^{n(n-1)/2} f\ldotp
  \end{equation}
  Use this and the fact that \eqref{eq:108} is onto to prove that
  $\HH^i(G, S^\ast) = 1$ for $i = 2, 4, 6, \dots$, proving (d).  Part
  (b) follows from (a), (d), and sequence \eqref{eq:112}. Part (e)
  follows from sequence \eqref{eq:112}.
\end{proof}
\begin{conjecture}
  If $f$ is irreducible, then $T^\ast = k^\ast$. A partial answer is
  given in Theorems~\ref{th:2.10} and ~\ref{th:2.12}.
\end{conjecture}
\begin{theorem}
  \label{th:2.10}
  In the context of Section~\ref{sec:2.2}, assume $n=p$ is a prime
  number.  Then $T^\ast = k^\ast$.
\end{theorem}
\begin{proof}
  This is a consequence of Theorem~\ref{th:2.3} and
  Theorem~\ref{th:2.8}(c).
\end{proof}
\begin{lemma}
  \label{lemma:2.11}
  Let $p$ be a prime and $n = p^2$.  Let $G = \langle\sigma\rangle$,
  $H = \langle\sigma^p\rangle$.  Let
  \[
  \begin{split}
    T_2 &= A[z]/(z^n - f), \\
    S_2 &=
    R[z]/(z^n - f),\\
    T_1 &= T_2^H = A[z^p],\\
    S_1 &= S_2^H = R[z^p]\ldotp
  \end{split}
  \]
  The following are true.
  \begin{enumerate}[(a)]
  \item
    \[
    \HH^i(H, S_2^\ast/T_2^\ast) =
    \begin{cases}
      \langle z\rangle & \text{if $i = 0$}\\
      \langle 1\rangle & \text{if $i = 1, 3, 5, \dots$}\\
      \langle z\rangle/\langle z^p\rangle & \text{if $i = 2, 4, 6,
        \dots$.}
    \end{cases}
    \]
  \item $\HH^0(H, T_2^\ast) = T_1^\ast = k^\ast$ and $\HH^0(H,
    S_2^\ast) = S_1^\ast = k^\ast\times \langle z^p\rangle$.
  \item If $i > 0$ is even, then $\HH^i(H, T_2^\ast) = 1$ and
    $\HH^i(H, S_2^\ast) = 1$.
  \item If $i > 0$ is odd, there is a short exact sequence
    \[
    1 \rightarrow \mu_p \rightarrow \HH^i(H, T_2^\ast) \rightarrow
    \HH^i(H, S_2^\ast) \rightarrow 1\ldotp
    \]
  \item The group $\HH^1(H, S_2^\ast)$ is isomorphic to
    $\Cl(S_2/S_1)$, which is the kernel of the natural map $i :
    \Cl(S_1) \rightarrow \Cl(S_2)$.
  \end{enumerate}
\end{lemma}
\begin{proof} {F}rom Theorem~\ref{th:2.8}, $S_2^\ast/T_2^\ast =
  \langle z\rangle$ with trivial $G$-action. This gives (a).  Since
  $T_2^H = T_1$, Theorem~\ref{th:2.10} implies $(T_2^\ast)^H =
  T_1^\ast = k^\ast$ and $(S_2^\ast)^H = S_1^\ast = k^\ast \times
  \langle z^p\rangle$, which is (b).  Since $H$ is cyclic, if $i > 0$
  is even, then $\HH^i(H, T_2^\ast)$ is equal to $k^\ast/ N T_2^\ast$,
  which is trivial for the same reason that the norm \eqref{eq:108} is
  onto.  As in \eqref{eq:130}, one can check that $z^p$ is in the
  image of the norm map $S_2^\ast \rightarrow S_1^\ast$.  Part (c)
  follows.  The exact sequence
  \begin{multline}
    \label{eq:131}
    1 \rightarrow (T_2^\ast)^H \rightarrow (S_2^\ast)^H \rightarrow
    S_2^\ast/T_2^\ast \rightarrow \HH^1(G, T_2^\ast) \rightarrow  \\
    \HH^1(H, S_2^\ast) \rightarrow \HH^1(H, S_2^\ast/T_2^\ast)
    \rightarrow \HH^2(H, T_2^\ast) \rightarrow \HH^2(H, S_2^\ast)
    \rightarrow \\
    \HH^2(H, S_2^\ast/T_2^\ast) \rightarrow \HH^3(H, T_2^\ast)
    \rightarrow \HH^3(H, S_2^\ast) \rightarrow \HH^3(H,
    S_2^\ast/T_2^\ast) \rightarrow \dots
  \end{multline}
  is the counterpart of \eqref{eq:128} for the group $H$. Use
  \eqref{eq:131}, parts (a), (b), (c), and periodicity to get (d). The
  groups in (d) are $\mathbb Z/p$-modules, so the sequence splits.
  Part (e) is proved in \cite[Theorem~16.1]{F:DCG}.
\end{proof}
\begin{theorem}
  \label{th:2.12}
  If $n = 4$, then $T^\ast = k^\ast$.
\end{theorem}
\begin{proof}
  In Lemma~\ref{lemma:2.11}(e), the group $H$ has order two. By
  \cite[Corollary~17.27]{MR2286236}, the group $\Cl(S_2/S_1)$ is
  annihilated by two. Hence $\Cl(S_2/S_1)$ is a subgroup of
  ${_2\Cl(S_1)}$.  By \cite[Theorem~2.3]{F:rBgadp}, ${_2\Cl(S_1)}$ is
  equal to $\Cl(S_1)^{G/H}$, which by Theorem~\ref{th:2.8}(b) is equal
  to $\langle 0\rangle$.  Lemma~\ref{lemma:2.11}(e) implies $\HH^1(H,
  S_2^\ast) = \langle 1\rangle$ and Lemma~\ref{lemma:2.11}(d) implies
  $\HH^1(H, T_2^\ast) \cong \mu_2$.  By Lemma~\ref{lemma:2.2},
  $\HH^1(H, T_2^\ast/k^\ast) = \langle 1 \rangle$.  As in the proof of
  Theorem~\ref{th:2.10}, $T_2^\ast/k^\ast = \langle 1 \rangle$.
\end{proof}
\begin{conjecture}
  If $n = 2^s$, for any $s > 0$, then $T^\ast = k^\ast$.  To iterate
  the argument used in Theorem~\ref{th:2.12}, it is necessary to know
  that elements of order two in $\Cl(S_i)$ are fixed by $G$.
\end{conjecture}

\section{Localization of a cyclic cover of projective space}
\label{sec:3}
In this section we consider the group of units on a ramified cyclic
cover $\pi : X \rightarrow U$ of affine varieties which is the
restriction of a cyclic cover of projective space $\pi:Y \rightarrow
\mathbb P^m$ to an open set.  We treat two special cases. In
Section~\ref{sec:3.1}, the localization of $\pi$ is such that along
the ``divisor at infinity'', $\pi$ is unramified.
Section~\ref{sec:3.3} considers the case where the ``divisor at
infinity'' is equal to the ramification divisor.  In
Section~\ref{sec:3.2}, these results are applied to the group of units
on an affine curve.  In Section~\ref{sec:3.4}, the results of this
section are applied to a special case of Example~\ref{example:1.4}.
Throughout Section~\ref{sec:3}, $k$ is an algebraically closed field
and if $G$ is the cyclic group whose action on $Y$ induces the
quotient morphism $\pi$, then we assume the order of $G$ is invertible
in $k$.

\subsection{Unramified at infinity}
\label{sec:3.1}
Start with a normal projective variety $Y$ of dimension $m > 0$,
together with $G = \langle \sigma \rangle$ a cyclic group of order $n$
acting on $Y$ such that the quotient morphism is $\pi : Y \rightarrow
\mathbb P^m$.  Next, we restrict $\pi$ to an affine open subset of
$\mathbb P^m$ which is the complement of a prime divisor that is split
by $\pi$.  Specifically, let $F \subseteq \mathbb P^m$ be an
irreducible hypersurface and $U = \mathbb P^m - F$ the affine open
complement.  Assume the irreducible hypersurface $F \subseteq\mathbb
P^m$ is split by $\pi$.  By this we mean $\pi : Y\times_{\mathbb P^m}
F \rightarrow F$ is unramified, and lying above $F$ are $n$
irreducible components.  If we write $\pi^{-1}(F) = F_1 \cup \dotsm
\cup F_n$, then for each $i$, $\pi : F_i \rightarrow F$ is an
isomorphism.  Taking $X$ to be $\pi^{-1}(U)$, the group $G$ acts on
$X$ and $\pi : X \rightarrow U$ is a ramified cyclic covering of
affine varieties.
The varieties defined so far make up the diagram:
\begin{equation}
  \label{eq:21}
  \begin{CD}
    X @>{\subseteq}>> Y @<{\supseteq}<<
        F_1+\dots +F_n \\
        @VV{\pi}V @VV{\pi}V @VV{\pi}V \\
    U @>{\subseteq}>>
    \mathbb P^m
    @<{\supseteq}<< F
  \end{CD}
\end{equation}
In this section we study the groups of units on the affine
varieties $U$ and $X$.
\begin{lemma}
  \label{lemma:3.1}
  \cite[Example~II.6.5.1]{H:AG} Let $F \subseteq\mathbb P^m$ be an
  irreducible hypersurface of degree $d$.  For the open affine $U =
  \mathbb P^m - F$,
  \begin{enumerate}[(a)]
  \item $\mathcal O^\ast(U)$, the group of units, is equal to
    $k^\ast$, and
  \item $\Cl(U)$, the class group, is a cyclic group of order $d$.
  \end{enumerate}
\end{lemma}
\begin{proof}
  The degree homomorphism $\deg : \Cl(\mathbb P^m) \rightarrow \mathbb
  Z$ is an isomorphism.  The exact sequence \eqref{eq:101} becomes
  \begin{equation}
    \label{eq:132}
    1 \rightarrow
    k^\ast
    \rightarrow
    \mathcal O^\ast(U)
    \xrightarrow{\Div}
    \mathbb Z F
    \xrightarrow{\chi}
    \Cl(\mathbb P^m)
    \rightarrow
    \Cl(U)
    \rightarrow 0
  \end{equation}
  where $\chi(F) = d$.  Therefore, $\chi$ is one-to-one.
\end{proof}
\begin{lemma}
  \label{lemma:3.2}
  In the context of Section~\ref{sec:3.1}, the following are true.
  \begin{enumerate}[(a)]
  \item $\left(\mathcal O^\ast(X)/k^\ast\right)^G = \langle 1\rangle$.
  \item The $\mathbb Z$-module $\mathcal O^\ast(X)/k^\ast$ is free and
    has rank less than or equal to $n-1$.
  \end{enumerate}
\end{lemma}
\begin{proof}
  Consider the commutative diagram
  \begin{equation}
    \label{eq:133}
    \begin{CD}
      {1 \rightarrow\mathcal O^\ast(U)/k^\ast} @>>> {\mathbb Z F} @>>>
      {\Cl(\mathbb P^m)} @>>>
      {\Cl(U) \rightarrow 0} \\
      @VV{\pi^\sharp}V @VV{\pi^\ast}V @VVV
      @VVV      \\
      {1\rightarrow\mathcal O^\ast(X)/k^\ast} @>{\Div}>>
      {\bigoplus_{i=1}^n \mathbb Z F_i} @>{\chi}>> {\Cl(Y)}
      @>>>{\Cl(X)\rightarrow 0}
    \end{CD}
  \end{equation}
  The top row of \eqref{eq:133} is \eqref{eq:132} and the second row
  is \eqref{eq:101}.  The map $\pi^\ast$ is defined by $F\mapsto
  F_1+\dots+F_n$.  The maps on class groups are the natural maps.  The
  $G$-module $\bigoplus_{i=1}^n \mathbb Z F_i$ is free of rank one.
  By \cite[Lemma~10.34]{R:IHA} the sequence
  \begin{equation}
    \label{eq:3}
    \bigoplus_{i=1}^n \mathbb Z F_i
    \xrightarrow{N}
    \bigoplus_{i=1}^n \mathbb Z F_i
    \xrightarrow{D}
    \bigoplus_{i=1}^n \mathbb Z F_i
    \xrightarrow{\deg}
    \mathbb Z \rightarrow 0
  \end{equation}
  of $G$-modules is exact, where $N$ and $D$ are as in
  Theorem~\ref{th:1.4} and $\deg(a_1 F_1+ \dots + a_n F_n) = {a_1 +
    \dots + a_n}$.  The image of $N$ is the cyclic $\mathbb Z$-module
  generated by $\pi^\ast(F) = F_1 + \dots + F_n$.  It follows that
  \begin{equation}
    \label{eq:134}
    \HH^i\Bigl(G,\bigoplus_{i=1}^n \mathbb Z F_i\Bigr) =
    \begin{cases}
      \langle \pi^\ast(F)\rangle = \langle
      F_1+\dots+F_n\rangle & \text{if $i=0$}\\
      (0) & \text{if $i > 0$.}
    \end{cases}
  \end{equation}
  Eq. \eqref{eq:3} breaks up, yielding short exact sequences
  \begin{gather}
    \label{eq:4}
    0 \rightarrow \image{(D)} \rightarrow \bigoplus_{i=1}^n \mathbb Z
    F_i \xrightarrow{\deg}
    \mathbb Z \rightarrow 0\\
    \label{eq:5}
    0 \rightarrow \image{(N)} \rightarrow \bigoplus_{i=1}^n \mathbb Z
    F_i \rightarrow \image{(D)} \rightarrow 0
  \end{gather}
  of $G$-modules.  As a sequence of $\mathbb Z$-modules, \eqref{eq:4}
  is split-exact, so $\image(D)$ is a free $\mathbb Z$-module of rank
  $n-1$.  A principal divisor on $Y$ has degree $0$ (see
  \cite[Exercise~II.6.2]{H:AG}), so in \eqref{eq:133}, the image of
  $\Div$ is a subgroup of the kernel of the map $\deg$ of
  \eqref{eq:4}.  Thus, $\image(\Div)$ is a free $\mathbb Z$-module of
  rank at most $n-1$, which proves (b).  The image of $N$ is $\langle
  \pi^\ast(F)\rangle$, which is a trivial $G$-module.
  Theorem~\ref{th:1.4} shows $\HH^1\left(G, \langle \pi^\ast(F)\rangle
  \right) = (0)$.  Using \eqref{eq:134}, the long exact sequence of
  cohomology associated to \eqref{eq:5} simplifies to
  \begin{equation}
    \label{eq:135}
    0 \rightarrow
    \langle \pi^\ast(F)\rangle
    \xrightarrow{=}
    \langle \pi^\ast(F)\rangle
    \rightarrow \left( \image{(D)}\right)^G
    \xrightarrow{\partial^0}
    0\ldotp
  \end{equation}
  Hence $\left( \image{(D)}\right)^G = (0)$.  Since $\image(\Div)
  \subseteq \image(D)$, this proves (a).
\end{proof}
\begin{remark}
  The upper bound of Lemma~\ref{lemma:3.2}(b) is sharp, as shown in
  Example~\ref{example:1.7}.
\end{remark}
\begin{theorem}
  \label{th:3.3}
  In the above context, assume $n= p$ is a prime number and $H$ is the
  subgroup of $\Cl(Y)$ generated by the prime divisors $F_1, \dots,
  F_n$.  For the group $\mathcal O^\ast(X)/k^\ast$, there are two
  possibilities.  The first is $\mathcal O^\ast(X) =k^\ast$, in which
  case $H$ is a free $\mathbb Z$-module of rank $p$.  The second
  possibility is that $\mathcal O^\ast(X)/k^\ast$ is a free $\mathbb
  Z$-module of rank $p-1$ and in this case $H$ is isomorphic to an
  extension of $\mathbb Z$ by a finite group.
\end{theorem}
\begin{proof}
  The second row of \eqref{eq:133} gives rise to the exact sequence
  \begin{equation}
    \label{eq:136}
    1 \rightarrow \mathcal O^\ast(X)/k^\ast \xrightarrow{\Div}
    \bigoplus_{i=1}^n \mathbb Z F_i \xrightarrow{\chi} H \rightarrow 0
  \end{equation}
  of $G$-modules.  As in Theorem~\ref{th:2.3}, the proof follows from
  Lemma~\ref{lemma:3.2}, the structure theory for a module over a
  cyclic group of order $p$, and \eqref{eq:136}.
\end{proof}

\subsection{The group of units on an affine curve}
\label{sec:3.2}
In this section we apply Theorem~\ref{th:3.3} to study the group of
units on an affine curve over $k=\mathbb C$, the field of complex
numbers.

Let $p$ be a prime number and $n \geq 2$ an integer such that $p \mid
n$.  Let $\lambda_1, \dots, \lambda_n$ be distinct elements of $k$ and
set $f(x) = (x-\lambda_1)\dotsm (x-\lambda_n)$.  Let $X = Z(y^p -
f(x))$, a nonsingular affine curve in $\mathbb A^2$. Let $\pi : X
\rightarrow \mathbb A^1$ be the morphism induced by $k[x] \rightarrow
\mathcal O(X)$.  Let $Y$ be the complete nonsingular model for $X$ and
$\pi : Y \rightarrow \mathbb P^1$ the extension of $\pi$.  Then $Y$ is
a cyclic cover of $\mathbb P^1$ of degree $p$, and we are in the
context of the introduction to Section~\ref{sec:3}.

Let $P_i$ denote the point on $X$ (and on $Y$) where $y = x-\lambda_i
= 0$.  Let $Q_i = \pi(P_i)$.  The map $\pi$ ramifies at $P_i$ and the
ramification index is $p$.  These are the only points where $\pi$ is
ramified.  Solve the Riemann-Hurwitz Formula
\cite[Corollary~IV.2.4]{H:AG} for the genus of $Y$ to get $g(Y) =
(p-1)(n-2)/2$.

If $n=p=2$, then $Y$ is a nonsingular plane curve of genus zero. It is
an exercise \cite[Exercise~I.1.1]{H:AG} to show that $\mathcal O(X)$
is isomorphic to $k[x,y]/(xy+1)$.  This is the special case of
Example~\ref{example:1.5} for which $n=2$. It follows that $\mathcal
O^\ast(X)/k^\ast$ is isomorphic to $\mathbb Z$. For the remainder of
Section~\ref{sec:3.2}, we will assume $n \geq 3$, hence $g(Y) \geq 1$.

The degree map $\deg : \Cl(Y) \rightarrow \mathbb Z$ is onto.  The
kernel of the degree map, denoted $\Cl^0(Y)$, is the jacobian variety.
By \cite[p.  64]{MR0282985}, $\Cl^0(Y)$ is an abelian variety of
dimension $g(Y)$.  By \cite[(iv), p. 42]{MR0282985}, $\Cl^0(Y)$ is a
divisible group.  Let $P_0$ be any closed point of $Y$.  The group of
divisors on $Y$ of degree zero, denoted $\bigDiv^0(Y)$, is generated
by the set $\{P-P_0 \mid P \in \bigDiv(Y)\}$.  In this context, the
exact sequence \eqref{eq:101} becomes
\begin{equation}
  \label{eq:137}
  1 \rightarrow   \mathcal O^\ast(Y)
  \rightarrow \mathcal O^\ast(Y - P_0)
  \rightarrow
  \mathbb Z P_0
  \xrightarrow{\chi}
  \Cl(Y)
  \rightarrow \Cl(Y -P_0)
  \rightarrow 0\ldotp
\end{equation}
Since $P_0$ has degree one, the degree map is a splitting for $\chi$,
so $\Cl(Y-P_0)$ is isomorphic to $\Cl^0(Y)$.

\begin{theorem}
  \label{th:3.4}
  Let $k = \mathbb C$ be the field of complex numbers and $p$ a prime
  number.  Say $G = \langle \sigma\rangle$ is a cyclic group of order
  $p$ acting on a nonsingular projective curve $Y$ such that the
  quotient map is $\pi: Y \rightarrow \mathbb P^1$.  Assume $Y$ has
  positive genus $g(Y) > 0$.  For a sufficiently general $Q\in \mathbb
  P^1$, if $X = \pi^{-1}(\mathbb P^1 - Q)$, then $\mathcal O^\ast(X) =
  k^\ast$.
\end{theorem}
\begin{proof}
  Let $P_0 \in Y$ be a point where $\pi$ is ramified. Let $Q_0 =
  \pi(P_0)$.  For an arbitrary $P \in Y$, the divisor $\pi(P) - Q_0$
  is a principal divisor on $\mathbb P^1$. Therefore,
  \begin{equation}
    \label{eq:138}
    \pi^\ast(\pi(P)-Q_0) =
    P + \sigma(P) + \dots + \sigma^{p-1}(P) - p P_0
  \end{equation}
  is a principal divisor on $Y$.  Manipulation of \eqref{eq:138} shows
  that
  \begin{equation}
    \label{eq:139}
    (P - \sigma(P))
    + 2 \sigma(P - \sigma(P)) 
    + 3 \sigma^2(P - \sigma(P)) 
    + \dots + p \sigma^{p-1}(P -  P_0)
  \end{equation}
  is a principal divisor on $Y$.  The group $\Cl^0(Y)$ is generated by
  the set of divisors $\{P-P_0 \mid P \in Y\}$ which is equal to the
  set $\{\sigma(P)-P_0 \mid P \in Y\}$.  By \eqref{eq:139} we see that
  $p \Cl^0(Y)$ is generated by the set of divisors $\{P- \sigma(P)
  \mid P \in Y\}$.  The group $\Cl^0(Y)$ is an abelian variety of
  dimension $g(Y)$ over $\mathbb C$. The subgroup of torsion elements
  in $\Cl^0(Y)$ make up a discrete group which is isomorphic to
  $(\mathbb Q/\mathbb Z)^{(2g)}$.  For a sufficiently general choice
  of $P$, $P-\sigma(P)$ generates an infinite subgroup of $\Cl^0(Y)$.
  Fix one such $P$.  Let $Q = \pi(P)$ and $X = Y - \pi^{-1}(Q) =
  \pi^{-1}(\mathbb P^1 - Q)$.  Let $H$ denote the subgroup of $\Cl(Y)$
  generated by $P, \sigma(P), \dots, \sigma^{p-1}(P)$.  We are in the
  context of Theorem~\ref{th:3.3}. Therefore $\mathcal O^\ast(X) =
  k^\ast$ if and only if $H$ is not an extension of $\mathbb Z$ by a
  finite group.  Consider the degree map $\deg : H \rightarrow \mathbb
  Z$, which splits since $\deg(P) = 1$.  By choice of $P$, the kernel
  of $\deg$ contains the infinite subgroup $\langle P- \sigma(P),
  \sigma(P)- \sigma^2(P), \dots, \sigma^{p-1}(P) - P \rangle$.  Since
  $H$ is not an extension of $\mathbb Z$ by a finite group,
  Theorem~\ref{th:3.3} says $H$ is a free $\mathbb Z$-module of rank
  $p$ and $\mathcal O^\ast(X) = k^\ast$.  The proof follows.
\end{proof}
\begin{proposition}
  \label{prop:3.5}
  Let $k = \mathbb C$ be the field of complex numbers.  Let $p$ be a
  prime number, $n \geq 3$ an integer such that $p \mid n$, and
  $\lambda_1, \dots, \lambda_n$ distinct elements in $k$.  Let $f(x) =
  \prod_{i=1}^n (x -\lambda_i)$ and $X = Z(y^p - f(x))$, an affine
  curve in $\mathbb A^2$.  For a sufficiently general choice of
  $\lambda_1, \dots, \lambda_n$, the group of units on $X$ is equal to
  $k^\ast$.
\end{proposition}
\begin{proof}
  Apply Theorem~\ref{th:3.4} to the nonsingular projective completion
  of $X$.
\end{proof}
\begin{proposition}
  \label{prop:3.6}
  Let $k = \mathbb C$, the field of complex numbers.  Let $n \geq 3$
  and $\lambda_1, \dots, \lambda_n$ distinct elements in $k$.  Let
  $f(x) = \prod_{i=1}^n (x -\lambda_i)$ and $X = Z(y^2 - f(x))$ the
  corresponding affine hyperelliptic curve in $\mathbb A^2$.  For a
  sufficiently general choice of $\lambda_1, \dots, \lambda_n$, the
  group of units on $X$ is equal to $k^\ast$.
\end{proposition}
\begin{proof}
  Let $Y$ be the nonsingular projective completion for $X$.  If $n$ is
  even, this follows from Proposition~\ref{prop:3.5}.  If $n$ is odd,
  then $Y -X$ consists of only one point and we apply \eqref{eq:137}.
\end{proof}

\begin{example}
  \label{example:3.7}
  In Proposition~\ref{prop:3.6}, if $X$ is arbitrary, the group of
  units can be non-trivial.  For instance, let $f = x^4 -1$. In
  $\mathcal O(X)$, we have $1 = x^4 - y^2$, hence $x^2 -y$ is
  invertible and $\mathcal O^\ast(X) \neq k^\ast$.  For another
  example, let $f = x^4 +x$ and $X = Z(y^2 - f)$. In $\mathcal O(X)$,
  \[
  \begin{split}
    \left((x^3 + 1/2) + x y \right) \left((x^3 + 1/2) - x y\right)
    &= (x^3 + 1/2)^2 - x^2 y^2 \\
    &= x^6 + x^3 + 1/4 - x^2(x^4+x) \\
    &= 1/4
  \end{split}
  \]
  Hence $\mathcal O^\ast(X) \neq k^\ast$.
\end{example}

\subsection{The ramification divisor is at infinity}
\label{sec:3.3}
Start with a normal projective variety $Y$ of dimension $m > 0$,
together with $G = \langle \sigma \rangle$ a cyclic group of order $n$
acting on $Y$ such that the quotient morphism is $\pi : Y \rightarrow
\mathbb P^m$.  Assume the ramification divisor of $\pi$ on $Y$ is the
set of prime divisors $Q = \{Q_1, \dots, Q_r\}$ and that the
ramification index at each $Q_i$ is $n$. Let $\pi(Q) = P = \{P_1,
\dots, P_r\}$.  Let $U = \mathbb P^m - P$ and $X = Y - Q$. Then $\pi:
X \rightarrow U$ is a Galois cover with group $G$.
The varieties defined so far make up the diagram:
\begin{equation}
\label{eq:22}
  \begin{CD}
    X @>{\subseteq}>> Y @<{\supseteq}<<
       Q = Q_1+\dots + Q_r \\
        @VV{\pi}V @VV{\pi}V @VV{\pi}V \\
    U @>{\subseteq}>>
    \mathbb P^m
    @<{\supseteq}<< P = P_1 + \dots+P_r
  \end{CD}
\end{equation}
The Nagata
sequences \eqref{eq:101} for $U$ and $X$ give the rows of the
commutative diagram:
\begin{equation}
  \label{eq:140}
  \begin{CD}
    {1\rightarrow\mathcal O^\ast(U)/k^\ast} @>>> {\bigoplus^r_{i=1}
      \mathbb Z P_i} @>>> {\Cl(\mathbb P^m)} @>>>
    {\Cl(U) \rightarrow 0}\\
    @VV{\pi^\sharp}V @VV{\delta}V
    @VVV @VVV \\
    {1\rightarrow\mathcal O^\ast(X)/k^\ast} @>>> {\bigoplus^r_{i=1}
      \mathbb Z Q_i} @>{\chi}>> {\Cl(Y)} @>>> {\Cl(X) \rightarrow 0}
  \end{CD}
\end{equation}
By the map $\delta$, $P_i$ is mapped to $nQ_i$ \cite[p. 30]{F:DCG}.
We consider the group of units $\mathcal O^\ast(X)/k^\ast$ and the
image of $\chi$, which is equal to the subgroup of $\Cl(Y)$ generated
by $Q_1, \dots, Q_r$.
\begin{proposition}
  \label{prop:3.9}
  In the above context, if $\Cl(U) = \langle 0\rangle$, then
  \begin{enumerate}[(a)]
  \item $\mathcal O^\ast(U)/k^\ast$ is a free $\mathbb Z$-module of
    rank $r-1$,
  \item $\mathcal O^\ast(X)/k^\ast$ is a free $\mathbb Z$-module of
    rank $r-1$, and
  \item $\HH^1(U, \mu_\nu)$ is a free $\mathbb Z/\nu$-module of rank
    $r-1$, for all positive integers $\nu$ such that $\nu$ is
    invertible in $k$.
  \end{enumerate}
\end{proposition}
\begin{proof}
  In the top row of \eqref{eq:140}, $\Cl(U) = \langle 0\rangle$ and
  $\Cl(\mathbb P^m) \cong \mathbb Z$.  The sequence
  \[
  1 \rightarrow {\mathcal O^\ast(U)/k^\ast} \rightarrow
  {\bigoplus^r_{i=1} \mathbb Z P_i} \rightarrow \Cl(\mathbb P^m)
  \rightarrow 0
  \]
  is split-exact, which gives (a).  Since $\pi^\sharp$ in
  \eqref{eq:140} is one-to-one, we know the rank of the free $\mathbb
  Z$-module ${\mathcal O^\ast(X)/k^\ast}$ is at least $r-1$.  On $Y$ a
  principal divisor has degree zero (see
  \cite[Exercise~II.6.2]{H:AG}), so the subgroup of $\Cl(Y)$ generated
  by the divisor $Q_1$ is infinite cyclic.  In the second row of
  \eqref{eq:140}, the image of $\chi$ is infinite.  Therefore, the
  kernel of $\chi$ is a free $\mathbb Z$-module of rank at most $r-1$,
  which proves (b).  By \cite[Corollary~18.5]{F:DCG}, $\Pic(U) =
  \langle 0\rangle$ since $\Cl(U) = \langle 0\rangle$.  The exact
  sequence \eqref{eq:12} simplifies to
  \begin{equation}
    \label{eq:6}
    1 \rightarrow
    {\mathcal O^\ast(U)}/k^\ast
    \xrightarrow{\nu}
    {\mathcal O^\ast(U)}/k^\ast
    \rightarrow \HH^1(U, \mu_\nu) \rightarrow 0\ldotp
  \end{equation}
  Part (c) follows from (a) and \eqref{eq:6}.
\end{proof}
\begin{proposition}
  \label{prop:3.10}
  Let $k$ be an algebraically closed field of characteristic $p$,
  where $p = 0$ is allowed. Let $n$ be a positive integer such that
  $n$ is invertible in $k$.  Let $\pi\colon X \rightarrow U$ be a
  cyclic Galois cover of degree $n$ of integral varieties with group
  $G$.  Then modulo groups of $p$-torsion, the sequence
  \[
  0 \rightarrow \mathbb Z/n \rightarrow \HH^1(U,\mathbb Q/\mathbb Z)
  \xrightarrow{\pi^\ast} \HH^1(X,\mathbb Q/\mathbb Z)^G \rightarrow 0
  \]
  is exact. The kernel of $\pi^\ast$ is generated by the class
  represented by the cyclic covering $\pi: X \rightarrow U$.
\end{proposition}
\begin{proof}
  We give the proof for $p=0$. For positive characteristic, reduce all
  groups modulo $p$-torsion.  The Hochschild-Serre spectral sequence
  \cite[p.~105]{M:EC} for $\pi: X \rightarrow U$ is
  \begin{equation}
    \label{eq:141}
    \HH^i(G,\HH^j(X,\mathbb Q/\mathbb Z)) \Rightarrow
    \HH^{i+j}(U,\mathbb Q/\mathbb Z) \ldotp
  \end{equation}
  Since $X$ is integral, the group of global sections of the constant
  sheaf is $\HH^0(X,\mathbb Q/\mathbb Z) = \mathbb Q/\mathbb Z$.
  Because $G$ acts trivially on $\mathbb Q/\mathbb Z$,
  $\HH^0(X,\mathbb Q/\mathbb Z)^G = \mathbb Q/\mathbb Z$. By
  Theorem~\ref{th:1.4}:
  \[
  \HH^i(G,\mathbb Q/\mathbb Z) =
  \begin{cases}
    \mathbb Q/\mathbb Z & \text{if $i=0$,}\\
    \langle 0\rangle & \text{if $i > 0$ is even,}\\
    \mathbb Z/n &\text{if $i$ is odd.}
  \end{cases}
  \]
  Consider the filtration $\HH^1(U,\mathbb Q/\mathbb Z) = E^1_0
  \supseteq E^1_1\supseteq 0$. Then
  \[
  E_1^1 = E^{1,0}_\infty = E_2^{1,0} = \HH^1(G,H^0(X,\mathbb Q/\mathbb
  Z)) \cong \mathbb Z/n
  \]
  and since ${d_2^{0,1}} : \HH^1(X,\mathbb Q/\mathbb Z)^G \rightarrow
  \HH^2(G,\mathbb Q/\mathbb Z)$ is the zero map, $E_0^1/E_1^1 =
  E_\infty^{0,1} = E_2^{0,1}= \HH^1(X,\mathbb Q/\mathbb Z)^G$.
  Therefore we have the short exact sequence
  \[
  0 \rightarrow \mathbb Z/n\rightarrow \HH^1(U,\mathbb Q/\mathbb Z)
  \rightarrow \HH^1(X,\mathbb Q/\mathbb Z)^G \rightarrow 0\ldotp
  \]
  The inclusion $E_1^1 \subseteq E^1_0$ corresponds to the subgroup of
  $\HH^1(U,\mathbb Q/\mathbb Z)$ generated by the cyclic Galois cover
  $\pi:X \rightarrow U$.
\end{proof}
\begin{proposition}
  \label{lemma:3.11}
  In the context of Section~\ref{sec:3.3}, the following are true.
  \begin{enumerate}[(a)]
  \item In the commutative square
    \[
    \begin{CD}
      \HH^1(\mathbb P^m, \mu_n) @>>> \HH^1(U, \mu_n)
      \\
      @VVV @VV{\pi^\ast}V \\
      \HH^1(Y, \mu_n) @>{\rho}>> \HH^1(X, \mu_n)
    \end{CD}
    \]
    induced by the left side of \eqref{eq:22}, the image of $\pi^\ast$
    is a subgroup of the image of $\rho$.  We say that $\pi^\ast$
    splits the ramification of each $\xi\in \HH^1(U, \mu_n)$.
  \item If $r \geq 2$ and $\Cl(U) = \langle 0\rangle$, the image of
    $\pi^\ast$ is isomorphic to $(\mathbb Z/n)^{(r-2)}$.
  \end{enumerate}
\end{proposition}
\begin{proof}
  By $\Sing{Y}$ we denote the singular locus of $Y$.  Since $Y$ is
  normal, the codimension of $\Sing{Y}$ in $Y$ is at least two. The
  codimension of $\Sing{Q}$ in $Y$ is at least two.  Since $X
  \rightarrow U$ is unramified, we know $X$ is nonsingular.  Let
  $\omega$ be a closed subset of $Q$, of codimension greater than or
  equal to two, such that $Y -\omega$ and $Q - \omega$ are
  nonsingular.  Because $\pi$ is a finite morphism, $\pi(\omega)$ has
  codimension at least two in $\mathbb P^m$.  By
  \cite[Lemma~VI.9.1]{M:EC}, $\HH^1(Y, \mu_n) = \HH^1(Y-\omega,
  \mu_n)$.  In this proof, we are going to utilize the fundamental
  classes of the divisors $Q_i-\omega \subseteq Y-\omega$ and
  $P_i-\pi(\omega) \subseteq \mathbb P^m - \pi(\omega)$.  For the
  background material, see Section~\ref{sec:1.2.1}.  For the two
  divisors we combine the counterparts of \eqref{eq:10} as in
  Diagram~\eqref{eq:140}.  The resulting square
  \begin{equation}
    \label{eq:143}
    \begin{CD}
      {\mathbb Z = \HH^1_{P_i-\pi(\omega)}(\mathbb P^m-\pi(\omega),
        \mathbb G_m)} @>>> {\HH^1(\mathbb P^m-\pi(\omega), \mathbb
        G_m)
        = \Cl(\mathbb P^m)}\\
      @V{\alpha_i}VV
      @VV{\pi^\ast}V \\
      {\mathbb Z = \HH^1_{Q_i-\omega}(Y-\omega, \mathbb G_m)}
      @>>>{\HH^1(Y-\omega, \mathbb G_m)=\Cl(Y)}
    \end{CD}
  \end{equation}
  commutes.  As in \eqref{eq:140}, $\alpha_i$ maps $1$ to $n$.  In the
  top row of \eqref{eq:143}, $1$ maps to the divisor class of $P_i$ in
  $\Cl(\mathbb P^m)$.  In the bottom row of \eqref{eq:143}, $1$ maps
  to the divisor class of $Q_i$ in $\Cl(Y)$.  Consider the commutative
  diagram
  \begin{equation}
    \label{eq:144}
    \begin{CD}
      {\mathbb Z = \HH^1_{P_i-\pi(\omega)}(\mathbb P^m-\pi(\omega),
        \mathbb G_m)} @>{\partial^1}>>
      {\HH^2_{P_i-\pi(\omega)}(\mathbb P^m-\pi(\omega), \mu_n) =
        \mathbb Z/n}\\
      @VV{\alpha_i}V
      @VV{\beta_i}V   \\
      {\mathbb Z = \HH^1_{Q_i-\omega}(Y-\omega, \mathbb G_m)}
      @>{\partial^1}>> {\HH^2_{Q_i-\omega}(Y-\omega, \mu_n)=\mathbb
        Z/n}
    \end{CD}
  \end{equation}
  where $\alpha_i$ is the left side of \eqref{eq:143} and the rows are
  from \eqref{eq:2}. Both of the connecting maps $\partial^1$ are
  onto.  Thus $\beta_i$ is the zero map.  The divisor $P -
  \pi(\omega)$ decomposes into the disjoint union $\cup_i (P_i -
  \pi(\omega))$.  Therefore the cohomology with supports decomposes
  into a direct sum:
  \[
  \HH^2_{P - \pi(\omega)}(\mathbb P^m-\pi(\omega), \mu_n) =
  \bigoplus_{i=1}^r \HH^2_{P_i-\pi(\omega)}(\mathbb P^m-\pi(\omega),
  \mu_n)\ldotp
  \]
  The similar decomposition occurs for $Q - \omega \subseteq
  Y-\omega$.  Consider the commutative diagram
  \begin{equation}
    \label{eq:145}
    \begin{CD}
      {\HH^1(\mathbb P^m-\pi(\omega), \mu_n)} @>>> {\HH^1(U, \mu_n)}
      @>{\partial^0}>> {\bigoplus_{i=1}^r
        \HH^2_{P_i-\pi(\omega)}(\mathbb P^m-\pi(\omega), \mu_n)}\\
      @VVV @VV{\pi^\ast}V @VV{\beta=\beta_1+\dots+\beta_r}V
      \\
      {\HH^1(Y-\omega, \mu_n)} @>{\rho}>> {\HH^1(X, \mu_n)} @>>>
      {\bigoplus_{i=1}^r \HH^2_{Q_i-\omega}(Y-\omega, \mu_n)}
    \end{CD}
  \end{equation}
  whose rows are from \eqref{eq:9}, hence are exact. The map $\beta$
  is the summation of the maps $\beta_i$ in \eqref{eq:144}, hence is
  the zero map.  This proves that the image of $\pi^\ast$ is a
  subgroup of the image of $\rho$, which gives (a).  In part (b),
  Proposition~\ref{prop:3.10} implies that the kernel of $\pi^\ast$ in
  \eqref{eq:145} is cyclic of order $n$.  By
  Proposition~\ref{prop:3.9}, the image of $\pi^\ast$ in
  \eqref{eq:145} is isomorphic to $(\mathbb Z/n)^{(r-2)}$.
\end{proof}
\begin{proposition}
  \label{prop:3.12}
  In the notation of Section~\ref{sec:3.3}, assume $\Cl(U) = \langle
  0\rangle$ and $r \geq 2$.
  \begin{enumerate}[(a)]
  \item The Galois cover $X \rightarrow U$ corresponds to adjoining
    the $n$th root of an invertible function on $U$.  That is, there
    exists a unit $f \in \mathcal O^\ast(U)$ such that $\mathcal O(X)$
    is isomorphic to $\mathcal O(U)[z]/(z^n- f)$.
  \item Let $H$ denote the image of $\chi$ in \eqref{eq:140}, which is
    the subgroup of $\Cl(Y)$ generated by $Q_1, \dots, Q_r$.  As an
    abelian group, $H$ is isomorphic to $\mathbb Z \oplus (\mathbb
    Z/n)^{(r-2)}$.
  \item The quotient group $\mathcal O^\ast(X)/\mathcal O^\ast(U)$ is
    cyclic of order $n$ and is generated by the element $z$ of
    Part~(a).
  \end{enumerate}
\end{proposition}
\begin{proof}
  We are given that $\Cl(U) = \langle 0\rangle$. The exact sequence of
  Kummer Theory \eqref{eq:11} implies any cyclic Galois extension of
  $\mathcal O(U)$ is of the form $\mathcal O(U)[\sqrt[n]{f}]$ for some
  $f \in \mathcal O^\ast(U)$. This is (a).  The diagram
  \begin{equation}
    \label{eq:13}
    \begin{CD}
      {1} @>>> {\mathcal O^\ast(U)/k^\ast} @>>> {\bigoplus^r_{i=1}
        \mathbb Z P_i} @>>> {\Cl(\mathbb P^m)}
      @>>> 0 \\
      @.  @VV{\pi^\sharp}V @VV{\delta}V
      @VV{\pi^\ast}V  \\
      {1} @>>> {\mathcal O^\ast(X)/k^\ast} @>{\Div}>>
      {\bigoplus^r_{i=1} \mathbb Z Q_i} @>{\chi}>> {H} @>>>
      {0} \\
      @.
      @VVV @VVV @VVV \\
      1 @>>> \frac{\mathcal O^\ast(X)}{\mathcal O^\ast(U)} @>>>
      {\bigoplus^r_{i=1} (\mathbb Z/n) Q_i} @>>>
      \frac{H}{\image(\pi^\ast)} @>>> {0}
    \end{CD}
  \end{equation}
  commutes, where the top two rows are from \eqref{eq:140} and are
  exact. The bottom row of \eqref{eq:13} is made up of the cokernels
  of the maps $\pi^\sharp$, $\delta$, and $\pi^\ast$.  Consider the
  map
  \begin{equation}
    \label{eq:16}
    {\mathcal O^\ast(X)/k^\ast} \xrightarrow{\Div}
    {\bigoplus^r_{i=1} \mathbb Z Q_i} 
  \end{equation}
  of the second row of \eqref{eq:13}.  By Proposition~\ref{prop:3.9}
  the groups $\mathcal O^\ast(U)/k^\ast$ and ${\mathcal
    O^\ast(X)/k^\ast}$ are free $\mathbb Z$-modules of rank $r-1$.  To
  prove (b), we show that the invariant factors of \eqref{eq:16} are
  $1, n, \dots, n$, where $n$ occurs with multiplicity $r-2$.  Using
  Part~(a), let $z \in \mathcal O^\ast(X)$ satisfy $z^n = f$.  By
  assumption the ramification index of $\pi$ at each $Q_i$ is $n$.
  Therefore $z$ is a local parameter at each $Q_i$.  This implies
  $\Div(z)$ generates a direct summand of ${\bigoplus^r_{i=1} \mathbb
    Z Q_i}$. This proves $1$ is an invariant factor of \eqref{eq:16}.
  Let $n_1, \dots, n_{r-2}$ denote the other invariant factors of
  \eqref{eq:16}.  The group ${H}$ contains the image of
  ${\HH^1(U,\mu_n)} \rightarrow {\HH^1(X,\mu_n)}$ which by
  Proposition~\ref{lemma:3.11} is a direct sum of $r-2$ copies of
  $\mathbb Z/n$. This proves $n$ is a divisor of $n_i$, for $1 \leq i
  \leq r-2$.

  Since $\pi^\sharp$ is one-to-one, its cokernel is finite.  Since
  $\delta$ is one-to-one and the groups on the bottom row of
  \eqref{eq:13} are finite, the Snake Lemma \cite[Theorem~6.5]{R:IHA}
  implies the kernel of $\pi^\ast$ is finite.  Since $\Cl(\mathbb P^m)
  \cong \mathbb Z$, we conclude that $\pi^\ast$ is one-to-one.  The
  Snake Lemma implies the bottom row of \eqref{eq:13} is exact. In
  particular, we know the groups ${\mathcal O^\ast(X)}/{\mathcal
    O^\ast(U)}$ and ${H}/{\image(\pi^\ast)}$ are both annihilated by
  $n$.  The functor $()\otimes_\mathbb Z \mathbb Q/\mathbb Z$ applied
  to the third column of \eqref{eq:13} gives the exact sequence
  \begin{equation}
    \label{eq:15}
    0 \rightarrow
    \Tor_1(H, \mathbb Q/\mathbb Z)
    \rightarrow
    \Tor_1(H/\image(\pi^\ast), \mathbb Q/\mathbb Z)
    \rightarrow
    \mathbb Q/\mathbb Z
  \end{equation}
  which shows the exponent of the torsion subgroup of $H$ is a divisor
  of $n$. This shows each invariant factor of \eqref{eq:16} satisfies
  $n_i = n$, for $1 \leq i \leq r-2$, which proves (b).

  Apply the functor $()\otimes_{\mathbb Z} \mathbb Z/n$ to the third
  column of \eqref{eq:13} to get the exact sequence
  \begin{multline}
    \label{eq:17}
    0 \rightarrow \Tor_1(H, \mathbb Z/n) \rightarrow
    \Tor_1(H/\image(\pi^\ast), \mathbb Z/n) \rightarrow \mathbb Z/n
    \xrightarrow{\pi^\ast\otimes 1} \\
    H \otimes \mathbb Z/n \rightarrow (H/\image{\pi^\ast}) \otimes
    \mathbb Z/n \rightarrow 0\ldotp
  \end{multline}
  In \eqref{eq:13} the image of $\pi^\ast$ is a subgroup of $nH$, so
  in \eqref{eq:17} the map $\pi^\ast\otimes 1$ is the zero map.  By
  Part~(b), $\Tor_1(H, \mathbb Z/n)$ is isomorphic to $(\mathbb
  Z/n)^{(r-2)}$.  Since $H/\image(\pi^\ast)$ is annihilated by $n$,
  the first part of \eqref{eq:17} reduces to the exact sequence
  \begin{equation}
    \label{eq:20}
    0 \rightarrow
    (\mathbb Z/n)^{(r-2)}
    \rightarrow
    H/\image(\pi^\ast)
    \rightarrow \mathbb Z/n \rightarrow 0
  \end{equation}
  which is split exact.  Eq. \eqref{eq:20} shows $H/\image(\pi^\ast)$
  is isomorphic to $(\mathbb Z/n)^{(r-1)}$.  It follows that the
  bottom row of \eqref{eq:13} is split exact.  Therefore, $\mathcal
  O^\ast(X)/\mathcal O^\ast(U)$ is a cyclic group of order $n$. By the
  computations above, we know the element $z$ is a generator for this
  group.  This proves (c).
\end{proof}

In the notation of Section~\ref{sec:3.3}, Proposition~\ref{prop:3.13}
is the $r=1$ case.
\begin{proposition}
  \label{prop:3.13}
  Let $P$ be an irreducible divisor on $\mathbb P^m$ of degree $n > 1$
  and $U = \mathbb P^m - P$.  Let $d>1$ be a divisor of $n$.  Let $X
  \rightarrow U$ be a cyclic Galois cover of degree $d$ corresponding
  to a generator of $\HH^1(U,\mu_d)$.  Then $\mathcal O^\ast(X) =
  k^\ast$.
\end{proposition}
\begin{proof}
  The Galois cover $X$ exists by Lemma~\ref{lemma:3.1} and Kummer
  Theory \eqref{eq:11}.  Let $\pi:Y \rightarrow \mathbb P^m$ be the
  cyclic cover which ramifies along $P$ (for example, see
  \cite[Section~I.17]{MR749574}). Let $\pi^{-1}(P) = Q$. In
  Diagram~\eqref{eq:140}, $\chi(Q)$ is non-zero because a principal
  divisor on $Y$ has degree zero (see \cite[Exercise~II.6.2]{H:AG}).
\end{proof}
\subsection{The group of units on an affine variety}
\label{sec:3.4}
In this section we apply results from Section~\ref{sec:3.3} to study
the group of units on the affine variety defined by an equation of the
form $f_1 \dotsm f_r = 1$, where the polynomials $f_i$ are distinct
irreducible forms in $k[x_1, \dots, x_m]$.
\begin{theorem}
\label{theorem:3.15}
Let $m \geq 2$ and $r \geq 2$. Let $f_1, \dots, f_r$ be irreducible
forms in $k[x_1, \dots, x_m]$ defining distinct varieties $Z(f_1),
\dots, Z(f_r)$ in $\mathbb A^m$.  Assume the set $\{\deg(f_1), \dots,
\deg(f_r)\}$ generates the unit ideal in $\mathbb Z$.  If $X$ denotes
the affine variety in $\mathbb A^m$ defined by $f_1\dotsm f_r = 1$,
then $\mathcal O^\ast(X)/k^\ast = \langle f_1\rangle \times \dotsm
\times \langle f_{r-1}\rangle$.
\end{theorem}
\begin{proof}
  For $i = 1, \dots, r$, let $d_i=\deg(f_i)$. Let $n = d_1 +\dots +
  d_r$.  Let $Y$ be the projective variety in $\mathbb P^{m} =
  \Proj{(k[x_0, x_1, \dots, x_m])}$ defined by $f_1\dotsm f_r =
  x_0^n$.  Let $Z_0 = Z(x_0)$ denote the hyperplane at infinity.  For
  each $i = 1, \dots, r$, let $F_i = Z(f_i)$ be the hypersurface in
  $\mathbb P^{m}$ defined by $f_i$ and set $L_i = F_i \cap Z_0$.  Note
  that $L_i$ refers to a subvariety of $Z_0$ as well as of $Y$.  The
  divisor at infinity on $Y$ is $L_1+ \dots+ L_r$.  We are given that
  $L_1, \dots, L_r$ are distinct prime divisors in $Z_0$.  View $Y$ as
  a cyclic cover of $Z_0 = \Proj{k[x_1, \dots, x_m]}$ and let $\pi : Y
  \rightarrow Z_0$ be the projection.  The ramification divisor of
  $\pi$ is the set $\{L_1, \dots, L_r\}$, and the ramification index
  of $\pi$ at each $L_i$ is one.  If we let $X$ be the open affine
  $Y-L_1-\dots - L_r$, then $\mathcal O(X) = k[x_1, \dots,
  x_m]/(f_1\dotsm f_r -1)$ is obtained by dehomogenizing with respect
  to $x_0$.  If we let $U = Z_0 - L_1 - \dots - L_r$, then $X
  \rightarrow U$ is a cyclic Galois covering.  The varieties defined
  so far make up the commutative diagram:
  \begin{equation}
    \label{eq:24}
    \begin{CD}
      Y = \Proj\left\{\frac{k[x_0, x_1, \dots, x_m]}{(f_1\dotsm f_r -
          x_0^n)}\right\} @<{\supseteq}<< X = \Spec\left\{\frac{k[x_1,
          \dots, x_m]}{(f_1\dotsm f_r -
          1)}\right\}  =  Y - L_1-\dots - L_r \\
      @VV{\pi}V @VV{\pi}V  \\
      Z_0 = \Proj{k[x_1, \dots, x_n]} @<{\supseteq}<< U = Z_0 - L_1 -
      \dots - L_r
    \end{CD}
  \end{equation}
  At the generic point of $Y$ we have
  \begin{equation}
    \label{eq:14}
    \frac{f_1}{x_0^{d_1}}\dotsm
    \frac{f_r}{x_0^{d_r}} = 1
  \end{equation}
  so the function $f_i/x_0^{d_i}$ represents a unit in the coordinate
  ring $\mathcal O(X)$.  The diagram
  \begin{equation}
    \label{eq:18}
    \begin{CD}
      1 @>>> \prod_{i=1}^r \left\langle {f_i}/{x_0^{d_i}}
      \right\rangle @>{\Div}>> {\bigoplus^r_{i=1} \mathbb Z
        \Div(f_i/x_0^{d_i})}
      @>>> 0 \\
      @. @VV{\alpha}V @V{\beta}VV @VVV \\
      1 @>>> \frac{\mathcal O^\ast(X)}{k^\ast} @>{\Div}>>
      {\bigoplus^r_{i=1} \mathbb Z L_i} @>{\chi}>> H @>>> 0
    \end{CD}
  \end{equation}
  commutes, where the second row comes from the middle row in
  \eqref{eq:13} and is exact.  Given \eqref{eq:14} and
  Proposition~\ref{prop:3.9}, it is enough to show $\alpha$ is onto.
  For this, we utilize the matrix of the map $\beta$ in \eqref{eq:18}.
  Since $Y$ is defined by $f_1\dotsm f_r = x_0^n$, each $F_i$
  intersects $Y$ along $L_i$ with intersection multiplicity $n$. It
  follows that the divisor of $f_i/x_0^{d_i}$ on $Y$ is
  \[
  \Div(f_i/x_0^{d_i}) = nL_i - d_1 L_1-\dots - d_r L_r \ldotp
  \]
  The matrix of the map $\beta$ is:
  \begin{equation}
    \label{eq:19}
    C =
    \begin{bmatrix}
      n-d_1 & -d_2 & \dots & -d_{r-1} & -d_r\\
      -d_1 & n-d_2 & \dots & -d_{r-1} & -d_r\\
      & &         \vdots & & \\
      -d_1 & -d_2 & \dots & n-d_{r-1} & -d_r \\
      -d_1 & -d_2 & \dots & -d_{r-1} & n-d_r
    \end{bmatrix}\ldotp
  \end{equation}
  We compute the invariant factors of the matrix $C$.  First, replace
  column $r$ with the sum of columns $1$, \dots, $r$.  Second, for $i =
  1, \dots, r-1$, replace row $i$ with row $i$ minus row $r$. After
  these column and row operations, the resulting matrix is
  \begin{equation}
    \label{eq:23}
    \begin{bmatrix}
      n &  0 & \dots & 0 & 0\\
      0 & n & \dots & 0 & 0\\
      & &         \vdots & & \\
      0 & 0 & \dots & n & 0 \\
      -d_1 & -d_2 & \dots & -d_{r-1} & 0
    \end{bmatrix}\ldotp
  \end{equation}
  The ideal in $\mathbb Z$ generated by $d_1, d_2, \dots, d_{r-1}, n$
  is the unit ideal.  Using the method of minors (see for example,
  \cite[Theorem~9.64]{MR2043445}), one can compute the invariant
  factors of $C$. They are $1$ and $0$, each with multiplicity one,
  and $n$ with multiplicity $r-2$.  Since $d_1, \dots, d_r$ are
  relatively prime, $\Cl(U) = \Cl(Z_0 - L_1-\dots- L_r) = 0$.  By
  Proposition~\ref{prop:3.12} we know that the group $H$ in
  \eqref{eq:18} is isomorphic to $\mathbb Z \oplus \left(\mathbb
    Z/n\right)^{(r-2)}$.  By the Snake Lemma applied to \eqref{eq:18},
  the sequence
  \begin{equation}
    \label{eq:25}
    0 \rightarrow \Cokernel{\alpha}
    \rightarrow \Cokernel{\beta}
    \rightarrow H \rightarrow 0
  \end{equation}
  is exact.  The second and third groups in \eqref{eq:25} have the
  same invariant factors.  The exact functor $()\otimes_Z \mathbb Q$
  applied to \eqref{eq:25} shows $\Cokernel{\alpha}$ is finite.  Apply
  the functor $()\otimes_Z \mathbb Q/ \mathbb Z$ to \eqref{eq:25}.
  The exact $\Tor$ sequence which results shows that the torsion
  subgroup of $\Cokernel{\beta}$ maps onto the torsion subgroup of
  $H$.  This proves $\Cokernel{\alpha} = \langle 0\rangle$. Therefore
  $\{f_1, \dots, f_{r-1}\}$ is a basis for the group $\mathcal
  O^\ast(X)/k^\ast$.
\end{proof}
\begin{example}
  \label{example:3.14}
  This is a special case of Example~\ref{example:1.4}, as well as
  Theorem~\ref{theorem:3.15}.  Let $m \geq 2$ and $n \geq 2$. Let
  $f_1, \dots, f_n$ be distinct linear forms in $k[x_1, \dots, x_m]$.
  Let $X$ be the affine variety in $\mathbb A^m$ defined by $f_1
  \dotsm f_n =1$.  Then $\mathcal O(X) = k[x_1, \dots, x_m]/(f_1\dotsm
  f_n -1)$.  By Theorem~\ref{theorem:3.15}, $\{f_1, \dots, f_{n-1}\}$
  is a basis for the group $\mathcal O^\ast(X)/k^\ast$.
\end{example}

\section*{Acknowledgment}
The author is sincerely grateful to the referee for a very thorough
report that was thoughtfully and graciously prepared.

%\bibliography{../bib/tim} %
%\bibliography{../../bib/tim} %
%\bibliographystyle{amsplain}

\def\cfudot#1{\ifmmode\setbox7\hbox{$\accent"5E#1$}\else
  \setbox7\hbox{\accent"5E#1}\penalty 10000\relax\fi\raise 1\ht7
  \hbox{\raise.1ex\hbox to 1\wd7{\hss.\hss}}\penalty 10000 \hskip-1\wd7\penalty
  10000\box7}
\providecommand{\bysame}{\leavevmode\hbox to3em{\hrulefill}\thinspace}
\providecommand{\MR}{\relax\ifhmode\unskip\space\fi MR }
\providecommand{\MRhref}[2]{%
  \href{http://www.ams.org/mathscinet-getitem?mr=#1}{#2}
}
\providecommand{\href}[2]{#2}

\end{document}